\documentclass[preprint]{imsart}

\RequirePackage[OT1]{fontenc}
\RequirePackage{amsthm,amsmath,amssymb}
\RequirePackage[colorlinks,citecolor=blue,urlcolor=blue]{hyperref}

\RequirePackage[numbers]{natbib}

\startlocaldefs
\usepackage{changebar}

\numberwithin{equation}{section}
\theoremstyle{plain}

\newtheorem{theorem}{Theorem}[section]
\newtheorem{lemma}{Lemma}[section]

\theoremstyle{remark}
\newtheorem{remark}{Remark}[section]
\endlocaldefs

\usepackage{graphicx}
\usepackage{subfigure}
\usepackage{amsmath,amsthm,bm,bbm}
\usepackage{rotating}
\usepackage{comment}

\usepackage{xcolor}         
\definecolor{grau}{rgb}{0.8,0.8,0.8}
\newcommand{\fzcolor}{\color{orange}}

\theoremstyle{definition}

\DeclareMathOperator*{\var}{{\mathrm{var}}}

\newcommand{\transpose}{^{\mathrm{T}}}

\newcommand{\leps}{{\underline{\epsilon}}}

\newcommand{\calA}{{\mathcal{A}}}

\newcommand{\calD}{{\mathcal{D}}}

\newcommand{\calF}{{\mathcal{F}}}
\newcommand{\calG}{{\mathcal{G}}}
\newcommand{\calH}{{\mathcal{H}}}

\newcommand{\calN}{{\mathcal{N}}}

\newcommand{\calP}{{\mathcal{P}}}

\newcommand{\calS}{{\mathcal{S}}}

\newcommand{\calX}{{\mathcal{X}}}
\newcommand{\calY}{{\mathcal{Y}}}

\newcommand{\bx}{{\mathbf{x}}}

\newcommand{\bw}{{\mathbf{w}}}

\newcommand{\bz}{{\mathbf{z}}}

\newcommand{\bbeta}{{\bm{\beta}}}

\newcommand{\eye}{{\mathbf{I}}}

\newcommand{\zero}{{\bm{0}}}

\newcommand{\eps}{\epsilon}

\begin{document}

\begin{frontmatter}
\title{A Theoretical Framework for Bayesian Nonparametric Regression}
\runtitle{A Theoretical Framework for Bayesian Regression}



\begin{aug}
\author{\fnms{Fangzheng} \snm{Xie}\thanksref{addr1} \ead[label=e1]{fxie5@jhu.edu}},
\author{\fnms{Wei} \snm{Jin}\thanksref{addr1}\ead[label=e1,mark]{wjin@jhu.edu}},
\and
\author{\fnms{Yanxun} \snm{Xu}\thanksref{addr1,addr2}\ead[label=e2]{yanxun.xu@jhu.edu}}

\runauthor{F. Xie, W. Jin and Y. Xu}

\address[addr1]{Department of Applied Mathematics and Statistics, Johns Hopkins University}
\address[addr2]{Correspondence should be addressed to 
				\printead{e2}}

\end{aug}

\begin{abstract}
We develop a unifying framework for Bayesian nonparametric regression to study the rates of contraction with respect to the integrated $L_2$-distance without assuming the regression function space to be uniformly bounded. The framework is very flexible and can be applied to a wide class of nonparametric prior models. Three non-trivial applications of the proposed framework are provided: The finite random series regression of an $\alpha$-H\"older function, with adaptive rates of contraction up to a logarithmic factor; The un-modified block prior regression of an $\alpha$-Sobolev function, with adaptive-and-exact rates of contraction; The Gaussian spline regression of an $\alpha$-H\"older function, with the near optimal posterior contraction. These applications serve as generalization or complement of their respective results in the literature. 
Extensions to the fixed-design regression problem and sparse additive models in high dimensions are discussed as well. 
\end{abstract}


\end{frontmatter}

\section{Introduction} 
\label{sec:introduction}
  Consider the standard nonparametric regression problem
  $y_i=f(\bx_i)+e_i$, $i=1,\cdots,n$,
  where the set of predictors $(\bx_i)_{i=1}^n$ are referred to as design (points) and take values in $[0,1]^p\subset\mathbb{R}^p$, $e_i$'s are independent and identically distributed (i.i.d.) {mean-zero Gaussian} noises with $\var(e_i)=\sigma^2$, and $y_i$'s are responses. We follow the popular Bayesian approach by assigning $f$
  a prior distribution, and perform inference tasks by finding the posterior distribution of $f$ given the observations $(\bx_i,y_i)_{i=1}^n$. 
  
  We propose a theoretical framework for Bayesian nonparametric regression to study the rates of contraction with respect to the integrated $L_2$-distance
  \[
  \|f-g\|_2=\left\{\int_{[0,1]^p} \left[f(\bx)-g(\bx)\right]^2\mathrm{d}\bx\right\}^{1/2}.\]
  The framework is very flexible and can be applied to a wide class of nonparametric prior models. In particular, we emphasize that it allows the space of regression functions to be unbounded, including the renowned Gaussian process priors as special examples.
  
  Rates of contraction of posterior distributions for Bayesian nonparametric priors have been studied extensively. Following the earliest framework on generic rates of contraction theorems with i.i.d. data proposed by \cite{ghosal2000convergence}, specific examples for density estimation via Dirichlet process mixture models \cite{canale2017posterior,ghosal2007posterior,ghosal2001entropies,shen2013adaptive} 
  and location-scale mixture models \cite{kruijer2010adaptive,xie2017bayesian} are discussed. For nonparametric regression, the rates of contraction had not been discussed until \cite{ghosal2007convergence}, who develop a generic framework for fixed-design nonparametric regression to study rates of contraction with respect to the empirical $L_2$-distance. There are extensive studies for various priors that fall into this framework, { including location-scale mixture priors \cite{de2010adaptive}, conditional Gaussian tensor-product splines \cite{de2012adaptive}, and Gaussian processes \cite{van2008rates,van2009adaptive}}, among which adaptive rates are obtained in \cite{de2010adaptive,de2012adaptive,van2009adaptive}. 

  Although it is interesting to achieve adaptive rates of contraction with respect to the empirical $L_2$-distance for nonparametric regression, this might be restrictive since the empirical $L_2$-distance quantifies the convergence of functions only at the given design points. In nonparametric regression, one also expects that the error between the estimated function and the true function can be globally small over the whole design space \cite{Xie2017KEMPOL}, \emph{i.e.}, small mean-squared error for out-of-sample prediction. Therefore the integrated $L_2$-distance is a natural choice. For Gaussian processes, \cite{vaart2011information, yang2017frequentist} provide contraction rates for nonparametric regression with respect to the integrated $L_2$ and $L_\infty$-distance, respectively. A novel spike-and-slab wavelet series prior is constructed in \cite{yoo2017adaptive} to achieve adaptive contraction with respect to the stronger $L_\infty$-distance. These examples however, take advantage of their respective prior structures and may not be easily generalized. 
  A closely related reference is \cite{JMLR:v16:pati15a}, which discusses the rates of contraction of the rescaled-Gaussian process prior for nonparametric random-design regression with respect to the integrated $L_1$-distance, which is weaker than the integrated $L_2$-distance. 
  Although a generic framework for the integrated $L_2$-distance is presented in \cite{huang2004convergence}, the prior there is imposed on a uniformly bounded function space and hence rules out some popular priors, \emph{e.g.}, the popular Gaussian process prior \cite{rasmussen2006gaussian}. 

  {It is therefore natural to ask the following fundamental question: for Bayesian nonparametric regression, can one build a unifying  framework to study rates of contraction for various priors with respect to the integrated $L_2$-distance without assuming the uniform boundedness of the regression function space?} In this paper we provide a positive answer to this question.  
  The major contribution of this work is that we prove the existence of an ad-hoc test function that is required in the generic rates of contraction framework in \cite{ghosal2000convergence} by leveraging  Bernstein's inequality and 
 imposing certain structural assumption on the sieves with  large prior probabilities. 
 This is made clear in Section \ref{sec:the_framework_and_main_results}.
  Furthermore, we do not require the prior to be supported on a uniformly bounded function space. 
  Consequently, we are able to establish a general rate of contraction theorem with respect to the integrated $L_2$-distance for Bayesian nonparametric regression. Examples of applications falling into this framework include the finite random series prior \cite{rivoirard2012posterior,scricciolo2006convergence}, the (un-modified) block prior \cite{gao2016rate}, and the Gaussian splines prior \cite{de2012adaptive}.
  {In particular, for the block prior regression, rather than modifying the block prior by conditioning on a truncated function space as in \cite{gao2016rate} with a known upper bound for the unknown true regression function, we prove that the un-modified block prior automatically yields rate-exact Bayesian adaptation for nonparametric regression without such a truncation.}
  We further extend the proposed framework to the fixed design regression and sparse additive models in high dimensions. 
  The analyses of the above applications and extensions under the proposed framework also generalize their respective results in the literature. These improvements and generalizations are made clear in Sections \ref{sec:applications} and \ref{sec:extensions}. 


  The layout of this paper is as follows. In Section \ref{sec:the_framework_and_main_results} we introduce the random series framework for Bayesian nonparametric regression and present the main result concerning 
  rates of contraction. As applications of the main result, we derive the rates of contraction of various renowned priors for nonparametric regression in the literature with substantial improvements in Section \ref{sec:applications}. 
  Section \ref{sec:extensions} elaborates on extensions of the proposed framework to the fixed design regression problem and sparse additive models in high dimensions. 
  The technical proofs of main results are deferred to Section \ref{sec:proofs}.

\subsection*{Notations}
  For $1\leq r\leq\infty$, we use $\|\cdot\|_r$ to denote both the $\ell_r$-norm on any finite dimensional Euclidean space and the integrated $L_r$-norm of a measurable function (with respect to the Lebesgue measure). In particular, for any function $f\in L_2([0,1]^p)$, we use $\|f\|_2$ to denote the integrated $L_2$-norm defined to be $\|f\|_2^2=\int_{[0,1]^p}f^2(\bx)\mathrm{d}\bx$. We follow the convention that when $r=2$, the subscript is omitted, \emph{i.e.}, $\|\cdot\|_2=\|\cdot\|$. 
  The Hilbert space $l^2$ denotes the space of sequences that are squared-summable. 
  We use $\lfloor x\rfloor$ to denote the maximal integer no greater than $x$, and $\lceil x\rceil$ to denote the minimum integer no less than $x$. The notations $a\lesssim b$ and $a\gtrsim b$ denote the inequalities up to a positive multiplicative constant, and we write $a\asymp b$ if $a\lesssim b$ and $a\gtrsim b$. 
  {Throughout capital letters $C, C_1, \tilde{C}, C', D, D_1, \tilde{D}, D', \cdots$ are used to denote generic positive constants and their values might change from line to line unless particularly specified, but are universal and unimportant for the analysis.} 

  We refer to $\calP$ as a statistical model if it consists of a class of densities on a sample space $\calX$ with respect to some underlying $\sigma$-finite measure. 
  Given a (frequentist) statistical model $\calP$ and the i.i.d. data $(\bw_i)_{i=1}^n$ from some $P\in\calP$, the prior and the posterior distribution on $\calP$ are always denoted by $\Pi(\cdot)$ and $\Pi(\cdot\mid\bw_1,\cdots,\bw_n)$, respectively. Given a function $f:\calX\to\mathbb{R}$, we use $\mathbb{P}_nf=n^{-1}\sum_{i=1}^nf(\bx_i)$ to denote the empirical measure of $f$, and $\mathbb{G}_nf=n^{-1/2}\sum_{i=1}^n\left[f(\bx_i)-\mathbb{E}f(\bx_i)\right]$ to denote the empirical process of $f$, given the i.i.d. data $(\bx_i)_{i=1}^n$. With a slight abuse of notations, when applying to a set of design points $(\bx_i)_{i=1}^n$, we also denote $\mathbb{P}_nf=n^{-1}\sum_{i=1}^nf(\bx_i)$ and $\mathbb{G}_nf = n^{-1/2}\sum_{i=1}^n\left[f(\bx_i)-\mathbb{E}f(\bx_i)\right]$ to be the empirical measure and empirical process, even when the design points $(\bx_i)_{i=1}^n$ are deterministic. 
  In particular, $\phi$ denotes the probability density function of the (univariate) standard normal distribution, and we use the shorthand notation $\phi_\sigma(y)=\phi(y/\sigma)/\sigma$ to denote the density of $\mathrm{N}(0,\sigma^2)$. For a metric space $(\calF, d)$, for any $\epsilon>0$, the $\epsilon$-covering number of $(\calF,d)$, denoted by $\calN(\epsilon,\calF,d)$, is defined to be the minimum number of $\epsilon$-balls of the form $\{g\in\calF:d(f,g)<\epsilon\}$ that are needed to cover $\calF$.

\section{The framework and main results} 
\label{sec:the_framework_and_main_results}

  Consider the nonparametric regression model: $y_i=f(\bx_i)+e_i$, where $(e_i)_{i=1}^n$ are i.i.d. mean-zero {Gaussian noises with $\var(e_i)=\sigma^2$}, and $(\bx_i)_{i=1}^n$ are design points taking values in $[0,1]^p$. Unless otherwise stated,
  the design points $(\bx_i)_{i=1}^n$ are assumed to be independently and uniformly sampled for simplicity throughout the paper. Our framework naturally adapts to the case where the design points are independently sampled from a density function that is bounded away from $0$ and $\infty$. We assume that the responses $y_i$'s are generated from $y_i=f_0(\bx_i)+e_i$ for some unknown $f_0\in L_2([0,1]^p)$, thus the data $\calD_n = (\bx_i,y_i)_{i=1}^n$ can be regarded as i.i.d. samples from a distribution $\mathbb{P}_0$ with joint density $p_0(\bx,y)=\phi_\sigma(y-f_0(\bx))$. 
  Throughout we assume that the variance $\sigma^2$ of the noises is known, but our framework can be easily extended to the case where $\sigma$ is unknown by placing a prior on $\sigma$ that is supported on a compact interval contained in $(0,\infty)$ with a density bounded away from $0$ and $\infty$
  (see, for example, Section 2.2.1 in \cite{de2010adaptive} and Theorem 3.3 in \cite{van2008rates}). 

  Before presenting the main result, let us first introduce the basic framework for studying convergence of Bayesian nonparametric regression. In the context of the aforementioned nonparametric regression, by assigning a prior $\Pi$ on the regression function $f$, one obtains the posterior distribution $\Pi(f\in\cdot\mid\calD_n)$ defined through
  \[
  \Pi(f\in A\mid\calD_n)=\frac{\int_A\prod_{i = 1}^n[p_f(\bx_i,y_i)/p_0(\bx_i,y_i)]\Pi(\mathrm{d}f)}{\int\prod_{i = 1}^n [p_f(\bx_i,y_i)/p_0(\bx_i,y_i)]\Pi(\mathrm{d}f)}
  \]
  for any measurable function class $A$, where $p_f(\bx, y) = \phi_\sigma(y - f(\bx))$.
  In order that the posterior distribution $\Pi(\cdot\mid\calD_n)$ contracts to $f_0$ at rate $\eps_n$ with respect to a distance $d$, \emph{i.e.},
  $\Pi(d(f,f_0)>M\eps_n\mid\calD_n)\to 0$
  in $\mathbb{P}_0$-probability for some large constant $M>0$, the authors of \cite{ghosal2000convergence}  proposed the following renowned sufficient conditions, referred to as the prior-concentration-and-testing framework: There exist some constants $D, D'>0$, such that for sufficiently large $n$:
  \begin{enumerate}
    \item The prior concentration condition holds:
    \begin{align}
    \label{eqn:prior_concentration_classical}
    \Pi\left(\mathbb{E}_0\left(\log\frac{p_0}{p_f}\right)\leq\eps_n^2, \mathbb{E}_0\left[\left(\log\frac{p_0}{p_f}\right)^2\right]\leq\eps_n^2\right)\geq\mathrm{e}^{-Dn\eps_n^2}.
    \end{align}
    \item There exists a sequence $(\calF_n)_{n=1}^\infty$ of subsets of $L_2([0, 1]^p)$ (often referred to as the sieves) and test functions $(\phi_n)_{n=1}^\infty$ such that $\Pi(\calF_n^c)\leq\mathrm{e}^{-(D+4)n\eps_n^2}$,
    \begin{align*}
    \mathbb{E}_0\phi_n\to 0,\text{ and }
    \sup_{f\in\calF_n\cap\{d(f,f_0)>M\eps_n\}}\mathbb{E}_f(1-\phi_n)
    \leq \mathrm{e}^{-D'Mn\eps_n^2}.
    \end{align*}
  \end{enumerate}
However, the above framework 
is not instructive for constructing the appropriate sieves $(\calF_n)_{n = 1}^\infty$ nor the desired test functions $(\phi_n)_{n = 1}^\infty$ for studying the rates of contraction for nonparametric regression with respect to $\|\cdot\|_2$. Specifically, it does not provide a guidance on how to construct the desired sieves, or what their structural features are. 
 The major contribution of this work, in contrast, is that we impose certain structural assumption on the sieves to construct the desired test functions. 
By doing so, we are able to modify the entire framework so that it can be applied to a variety of nonparametric regression priors for us 
  to derive the corresponding posterior contraction rates. 

  The following local testing lemma is the first technical contribution of this work. It also serves as a building block to construct the desired test functions required in the prior-concentration-and-testing framework. 
  \begin{lemma}\label{lemma:local_testability}
  For any $m\in\mathbb{N}_+$ and $\delta > 0$, assume the class of functions $\calF_m(\delta)$ satisfies
  \begin{align}\label{eqn:sieve_property}
  f\in\calF_{m}(\delta)
  \Longrightarrow
  \|f - f_0\|_\infty^2 \leq \eta (m\|f - f_0\|_2^2 + \delta^2)
  \end{align}
  for some constant $\eta > 0$. Then for any $f_1\in\calF_m(\delta)$ with $\sqrt{n}\|f_1-f_0\|_2>1$, there exists a test function $\phi_n:(\calX\times\calY)^n\to[0,1]$ such that
    \begin{align}
    \mathbb{E}_0\phi_n&\leq\exp\left(-Cn\|f_1-f_0\|_2^2\right),\nonumber\\
    \sup_{\{f\in\calF_m(\delta):\|f-f_1\|_2 \leq\xi \|f_0-f_1\|_2\}}\mathbb{E}_f(1-\phi_n)&\leq
    \exp\left(-Cn\|f_1-f_0\|_2^2\right)
    \nonumber\\&\quad
    +2
    \exp\left(-\frac{Cn\|f_1-f_0\|^2_2}{m\|f_1-f_0\|_2^2+\delta^2}\right)\nonumber
    \end{align}  
    for some constant $C>0$ and $\xi\in(0,1)$.
  \end{lemma}
  The key ingredient of Lemma \ref{lemma:local_testability} is the assumption \eqref{eqn:sieve_property} on the sieve $\calF_m(\delta)$. 
  By requiring that functions in $\calF_m(\delta)$ cannot explode in $\|f - f_0\|_\infty$ when $\|f - f_0\|_2$ is small, we can utilize Bernstein's inequality and obtain exponentially small type I and type II error probability bounds. 
  Based on Lemma \ref{lemma:local_testability}, we are able to establish the following  global testing lemma. 
  \begin{lemma}
  \label{lemma:global_testability}
  Suppose that $\calF_m(\delta)$ satisfies \eqref{eqn:sieve_property} for an $m\in\mathbb{N}_+$ and a $\delta>0$. Let $(\eps_n)_{n=1}^\infty$ be a sequence with $n\eps_n^2\to\infty$. Then for any $M\geq 0$, there exists a sequence of test functions $(\phi_n)_{n=1}^\infty$ such that
  \begin{align}
  \mathbb{E}_0\phi_n&\leq\sum_{j= M}^\infty N_{nj}\exp\left(-Cnj^2\eps_n^2\right), \nonumber\\
  \sup_{\{f\in\calF_m(\delta):\|f-f_0\|_2>M\eps_n\}}\mathbb{E}_f(1-\phi_n)
  &\leq \exp(-CM^2n\eps_n^2)
  +2\exp\left(-\frac{CM^2n\eps_n^2}{mM^2\eps_n^2+\delta^2}\right)\nonumber,
  \end{align}
  where $N_{nj}=\calN(\xi j\eps_n,\calS_{nj}(\eps_n),\|\cdot\|_2)$ is the covering number of 
  \[\calS_{nj}(\eps_n)=\left\{f\in\calF_m(\delta):j\eps_n<\|f-f_0\|_2\leq (j+1)\eps_n\right\},\] and $C$ is some positive constant.
  \end{lemma}
  The prior concentration condition \eqref{eqn:prior_concentration_classical} is  very important  in the study of Bayes theory. It guarantees that the denominator appearing in the posterior distribution $\int [p_f(\bx_i,y_i)/p_0(\bx_i,y_i)]\Pi(\mathrm{d}f)$ can be lower bounded by $\mathrm{e}^{-D'n\eps_n^2}$ for some constant $D'>0$ with large probability (see, for example, lemma 8.1 in \cite{ghosal2000convergence}). In the context of normal regression, the Kullback-Leibler divergence is proportional to the integrated $L_2$-distance between two regression functions. Motivated by this observation, we establish the following lemma that yields an exponential lower bound for the denominator 
  \[\int\prod_{i = 1}^n\frac{p_f(\bx_i,y_i)}{p_0(\bx_i,y_i)}\Pi(\mathrm{d}f)\]
 in the posterior distribution under the proposed framework. 
  \begin{lemma}
  \label{lemma:evidence_LB}
  Denote
  \[
  B(m,\eps,\omega)=\left\{f:\|f-f_0\|_2<\eps,\|f - f_0\|_\infty ^ 2\leq \eta'(m\|f - f_0\|_2^2 + \omega^2)\right\}
  \]
  for any $\eps,\delta>0$ and $m\in\mathbb{N}_+$, where $\eta' > 0$ is some constant. 
  {Suppose sequences $(\eps_n)_{n=1}^\infty$ and $(k_n)_{n=1}^\infty$ satisfy $\eps_n\to0$, $n\eps_n^2\to\infty$, $k_n\eps_n^2=O(1)$, and $\omega$ is some constant. }Then for any constant $C>0$, 
  \begin{align}
  \mathbb{P}_0\left(\int\prod_{i = 1}^n\frac{p_f(\bx_i,y_i)}{p_0(\bx_i,y_i)}\Pi(\mathrm{d}f)\leq {\Pi\left(B(k_n,\eps_n,\omega)\right)} \exp\left[-\left(C+\frac{1}{\sigma^2}\right)n\eps_n^2\right]\right)\to 0\nonumber.
  \end{align}
  \end{lemma}
  In some cases it is also straightforward to consider the prior concentration with respect to the stronger $\|\cdot\|_\infty$-norm. For example, for a wide class of Gaussian process priors, the prior concentration $\Pi(\|f-f_0\|_\infty<\eps)$ has been extensively studied (see, for example, \cite{ghosal2017fundamentals,van2008rates,van2009adaptive} for more details). 
  {
  \begin{lemma}
  \label{lemma:evidence_LB_infinity_norm}
  Suppose the sequence $(\eps_n)_{n=1}^\infty$ satisfies $\eps_n\to 0$ and $n\eps_n^2\to\infty$. Then for any constant $C>0$, 
  \begin{align}
  \mathbb{P}_0\left(\int\frac{p_f(\bx_i,y_i)}{p_0(\bx_i,y_i)}\Pi(\mathrm{d}f)\leq \Pi(\|f-f_0\|_\infty<\eps_n) \exp\left[-\left(C+\frac{1}{\sigma^2}\right)n\eps_n^2\right]\right)\to 0\nonumber.
  \end{align}
  \end{lemma}}
  Now we present the main result regarding the rates of contraction for Bayesian nonparametric regression. The proof is based on the modification of the prior-concentration-and-testing procedure. We also remark that the prior $\Pi$ on $f$ is not necessarily supported on a uniformly bounded function space, which is one of the major contributions of this work. 
  \begin{theorem}[Generic Contraction]\label{thm:generic_contraction}
  Let $(\eps_n)_{n=1}^\infty$ and $(\leps_n)_{n=1}^\infty$ be sequences such that $\min(n\eps_n^2,n\leps_n^2)\to\infty$ as $n\to\infty$, with $0\leq \leps_n\leq\eps_n\to 0$. Assume that the sieve $(\calF_{m_n}(\delta))_{n=1}^\infty$ satisfies \eqref{eqn:sieve_property} with $m_n\eps_n^2\to0$ for some constant $\delta$. In addition, assume that there exists another sequence $(k_n)_{n=1}^\infty\subset\mathbb{N}_+$ such that $k_n\leps_n^2=O(1)$. Suppose the following conditions hold for some constants $\omega, D>0$ and sufficiently large $n$ and $M$:
  \begin{align}
  \label{eqn:summability}
  &\sum_{j= M}^\infty N_{nj}\exp\left(-{Dnj^2\eps_n^2}\right)\to 0,\\
  \label{eqn:prior_mass_on_sieve}
  &\Pi(\calF_{m_n}^c(\delta))\lesssim\exp\left[-{\left(2D+\frac{1}{\sigma^2}\right)n\leps_n^2}\right],\\
  \label{eqn:prior_concentration}
  &{\Pi\left(B(k_n,\leps_n,\omega)\right)\geq\exp\left(-{Dn\leps_n^2}\right)},
  \end{align}
  where $N_{nj}=\calN(\xi j\eps_n,\calS_{nj}(\eps_n),\|\cdot\|_2)$ is the covering number of 
  \[\calS_{nj}=\{f\in\calF_{m_n}(\delta):j\eps_n<\|f-f_0\|_2\leq(j+1)\eps_n\},\] and $B(k_n,\leps_n,\omega)$ is defined in Lemma \ref{lemma:evidence_LB}. Then 
  $\mathbb{E}_0\left[\Pi(\|f-f_0\|_2>M\eps_n\mid\calD_n)\right]\to 0.$
  \end{theorem}
  \begin{remark}\label{remark:infinity_norm_concentration}
  In light of Lemma \ref{lemma:evidence_LB_infinity_norm}, by exploiting the proof of theorem \ref{thm:generic_contraction} we remark that when the assumptions and conditions in theorem \ref{thm:generic_contraction} hold with \eqref{eqn:prior_concentration} replaced by $\Pi(\|f-f_0\|_\infty<\leps_n)\geq\exp(-Dn\leps_n^2)$, the same rate of contraction also holds:
  $\mathbb{E}_0\left[\Pi(\|f-f_0\|_2>M\eps_n\mid\calD_n)\right]\to 0$
  for sufficiently large $M>0$. 
  \end{remark}

\section{Applications} 
\label{sec:applications}
In this section we consider three concrete priors on $f$ for the nonparametric regression problem $y_i=f(x_i)+e_i$, $i=1,\cdots,n$. For simplicity the design points are assumed to independently follow the one-dimensional uniform distribution $\mathrm{Unif}(0, 1)$. For some of the examples, the results can be easily generalized to the case where the design points are multi-dimensional by considering the tensor-product basis functions. 
{The  results  under the proposed framework also generalize their respective counterparts in the literature. }

\subsection{Finite random series regression with adaptive rate} 
\label{sub:sieve_prior_regression_with_adaptive_rate}
The finite random series prior \cite{arbel2013bayesian,rivoirard2012posterior,scricciolo2006convergence,doi:10.1111/sjos.12159} is  popular  in the literature of Bayesian nonparametric theory. It is a class of hierarchical priors that first draw an integer-valued random variable serving as the number of ``terms'' to be used in a finite sum, and then sample the ``term-wise'' parameters given the number of ``terms''. The finite random series prior typically does not depend on the smoothness level of the true function, and often yields minimax-optimal rates of contraction (up to a logarithmic factor) in many  nonparametric problems (\emph{e.g.}, density estimation \cite{rivoirard2012posterior,scricciolo2006convergence} and fixed-design regression \cite{arbel2013bayesian}).  
However, the adaptive rates of contraction for the finite random series prior in the random-design regression with respect to the integrated $L_2$-distance has not been established. In this subsection we address this issue by leveraging the framework proposed in Section \ref{sec:the_framework_and_main_results}.

We first introduce the finite random series prior. Let $(\psi_k)_{k = 1}^\infty$ be the Fourier basis in $L_2([0,1])$, \emph{i.e.}, $\psi_1(x) = 1,\psi_{2k}(x) = \sin k\pi x$, and $\psi_{2k + 1}(x) = \cos k\pi x$, $k\in\mathbb{N}_+$. Writing $f$ in terms of the Fourier series expansion
  $f(\bx)=\sum_{k = 1}^\infty\beta_k\psi_k(\bx)$, we then assign the finite random series prior $\Pi$ on $f$ by considering the following prior distribution on the coefficients $(\beta_k)_{k = 1}^\infty$: first sample an integer-valued random variable $N$ from a density function $\pi_N$ (with respect to the counting measure on $\mathbb{N}_+$), and then given $N = m$ the coefficients $\beta_k$'s are independently sampled according to
\[
\Pi(\mathrm{d}\beta_k\mid N = m) = 
\left\{
\begin{array}{ll}
g(\beta_k)
\mathrm{d}\beta_k,&\quad\text{if }1\leq k\leq m,\\
\delta_0(\mathrm{d}\beta_k),&\quad\text{if } k\geq m,
\end{array}
\right.
\]
where $g$ is an exponential power density $g(x)\propto \exp(-\tau_0|x|^\tau)$ for some $\tau,\tau_0>0$ \cite{scricciolo2011posterior}.
 We further require that 
\begin{align}\label{eqn:sieve_prior_number_of_terms}
\pi_N(m)\geq\exp(-b_0m\log m)\quad\text{and}\quad\sum_{N=m+1}^\infty\pi_N(N)\leq \exp(-b_1m\log m)
\end{align}
for some constants $b_0$, $b_1$. {The zero-truncated Poisson distribution $\pi_N(m)=(\mathrm{e}^\lambda-1)^{-1}\lambda^m/m!\mathbbm{1}(m\geq1)$ satisfies condition \eqref{eqn:sieve_prior_number_of_terms} \cite{xie2017bayesian}.} 


The true regression function $f_0$ is assumed to yield a Fourier expansion $f_0(x) = \sum_{k = 1}^\infty\beta_{0k}\psi_k(x)$ with regard to $(\psi_k)_{k=1}^\infty$, and be in the 
$\alpha$-H\"older ball
\begin{align}
\mathfrak{C}_\alpha(Q)=\left\{f(x)=\sum_{k=1}^\infty\beta_k\psi_k(x):\sum_{k=1}^\infty k^\alpha|\beta_k|\leq Q\right\},\nonumber
\end{align}
where
$\alpha>1/2$ is the smoothness level, and $Q>0$ is the $\alpha$-H\"older radius. 
Note that the construction of the aforementioned finite random series prior does not require the knowledge of the smoothness level $\alpha$. In the literature of Bayes theory, such a procedure is referred to as \emph{adaptive}. 


%

The following theorem shows that the constructed finite random series prior  is adaptive and the rate of contraction $n^{-\alpha/(2\alpha + 1)}(\log n)^t$ with respect to the integrated $L_2$-distance is minimax-optimal up to a logarithmic factor \citep{stone1982optimal}. 

\begin{theorem}\label{thm:sieve_prior_contraction}
Suppose the true regression function
$f_0\in\mathfrak{C}_\alpha(Q)$ for some $\alpha>1/2$ and $Q>0$, and $f$ is imposed the prior $\Pi$ given above.
Then there exists some sufficiently large constant $M>0$ such that 
\[
\mathbb{E}_0\left[\Pi\left(\|f-f_0\|_2>Mn^{-\alpha/(2\alpha+1)}(\log n)^{t}\mid\calD_n\right)\right]\to 0
\]
for any $t>\alpha/(2\alpha+1)$. 
\end{theorem}


\subsection{Block prior regression with adaptive and exact rate} 
\label{sub:block_prior_regression_with_adaptive_and_exact_rate}
In the literature of adaptive Bayesian procedure, the minimax-optimal rates of contraction are often obtained with an extra logarithmic factor. It typically requires extra work to obtain the exact minimax-optimal rate. 
Gao and Zhou \citep{gao2016rate} elegantly construct a modified block prior that yields rate-adaptive (\emph{i.e.}, the prior does not depend on the smoothness level) and rate-exact contraction (\emph{i.e.}, the contraction rate does not involve an extra logarithmic factor) for a wide class of nonparametric problems. 
Nevertheless, for nonparametric regression, \cite{gao2016rate} modifies the block prior by conditioning on the space of uniformly bounded functions. Requiring a known upper bound for the unknown $f_0$ when constructing the prior is restrictive since it eliminates the popular Gaussian process priors. Besides the theoretical concern, the block prior itself is also a conditional Gaussian process and such a modification is inconvenient for implementation.
In this section, we address this issue by showing that for nonparametric regression such a modification is not necessary. 

Recall that in Section \ref{sub:sieve_prior_regression_with_adaptive_rate} we have introduced the Fourier basis functions 
$(\psi_k)_{k = 1}^\infty$ in $L_2([0,1])$ with $\psi_1(x) = 1$, $\psi_{2k}(x) = \sqrt{2}\sin\pi kx$, and $\psi_{2k+1}(x) = \sqrt{2}\cos\pi kx$, $k\in\mathbb{N}_+$. 
The true regression function $f_0$ is also assumed to yield a Fourier expansion $f_0(x) = \sum_{k = 1}^\infty \beta_{0k}\psi_k(x)$, and to be in the $\alpha$-Sobolev ball
\[
\calH_\alpha(Q)=\left\{f(x)=\sum_{k=1}^\infty\beta_k\psi_k(x):\sum_{k=1}^\infty k^{2\alpha}\beta_k^2\leq Q\right\}\nonumber
\]
with radius $Q > 0$. 
%
Write $f(x) = \sum_{k = 1}^\infty \beta_k\psi_k(x)$ in terms of the Fourier expansion. Similar to the finite random series prior, the block prior is constructed by assigning a prior distribution on the coefficients $(\beta_k)_{k = 1}^\infty$ as follows. Given a sequence $\bbeta=(\beta_1,\beta_2,\cdots)$ in the squared-summable sequence space $l^2$, define the $\ell$th block $B_\ell$ to be the integer index set $B_\ell = \{k_\ell,\cdots,k_{\ell+1}-1\}$ and $n_\ell = |B_\ell| = k_{\ell+1}-k_\ell$, where $k_\ell = \lceil\mathrm{e}^\ell\rceil$. We use $\bbeta_\ell = (\beta_j:j\in B_\ell)\in\mathbb{R}^{n_\ell}$ to denote the coefficients with indices lying in the $\ell$th block $B_\ell$. The block prior then assigns the following distribution on the coefficients $(\beta_k)_{k=1}^\infty$: 
\begin{align}
\bbeta_\ell\mid A_\ell\sim\mathrm{N}(\zero, A_\ell\eye_{n_\ell}),\quad A_\ell\sim g_\ell,\quad\text{independently for each }\ell\nonumber,
\end{align}
where $(g_\ell)_{\ell=0}^\infty$ is a sequence of densities satisfying the following properties:
\begin{enumerate}
  \item There exists $c_1>0$ such that for any $\ell$ and $t\in[\mathrm{e}^{-\ell^2},\mathrm{e}^{-\ell}]$, 
  \begin{align}
  \label{eqn:block_prior_condition1}
  g_\ell(t)\geq\exp(-c_1\mathrm{e}^\ell).
  \end{align}
  \item There exists $c_2>0$ such that for any $\ell$, 
  \begin{align}
  \label{eqn:block_prior_condition2}
  \int_0^\infty tg_\ell(t)\mathrm{d}t\leq 4\exp(-c_2\ell^2).
  \end{align}
  \item There exists $c_3>0$ such that for any $\ell$,
  \begin{align}
  \label{eqn:block_prior_condition3}
  \int_{\mathrm{e}^{-\ell^2}}^\infty g_\ell(t)\mathrm{d}t\leq \exp(-c_3\mathrm{e}^\ell).
  \end{align}
\end{enumerate}
The existence of a sequence of densities $(g_\ell)_{\ell=0}^\infty$ satisfying \eqref{eqn:block_prior_condition1}, \eqref{eqn:block_prior_condition2}, and \eqref{eqn:block_prior_condition3} is verified in \cite{gao2016rate} (see proposition 2.1 in \cite{gao2016rate}).

Our major improvement for the block prior regression is the following theorem, which shows that the (un-modified) block prior yields rate-exact Bayesian adaptation for nonparametric regression. 
\begin{theorem}\label{thm:block_prior_contraction}
Suppose the true regression function
$f_0\in\calH_\alpha(Q)$ for some $\alpha>1/2$ and $Q>0$, and $f(x)=\sum_{k=1}^\infty\beta_k\psi_k(x)$ is imposed the block prior $\Pi$ as described above.
Then 
\[
\mathbb{E}_0\left[\Pi\left(\|f-f_0\|_2>Mn^{-\alpha/(2\alpha+1)}\mid\calD_n\right)\right]\to 0
\]
for some sufficiently large constant $M>0$.
\end{theorem}
Rather than using the sieve $\calF_n$ proposed in theorem 2.1 in \cite{gao2016rate}, which does not necessarily satisfy \eqref{eqn:sieve_property}, we construct $\calF_{m_n}(\delta)$ in a slightly different fashion:
\[
\calF_{m_n}(Q)=\left\{f(x)=\sum_{k=1}^\infty\beta_k\psi_k(x):\sum_{k=1}^\infty(\beta_k-\beta_{0k})^2k^{2\alpha}\leq Q^2\right\}
\]
with $m_n\asymp n^{1/(2\alpha+1)}$ and $\delta = Q$. The covering number $N_{nj}$ can be bounded using the metric entropy for Sobolev balls (for example, see lemma 6.4 in \cite{belitser2003adaptive}), and the rest conditions in \ref{thm:generic_contraction} can be verified using similar techniques as in \cite{gao2016rate}.

As discussed in Section 4.2 in \cite{gao2016rate}, the block prior can be easily extended to the wavelet basis functions and wavelet series. The wavelet basis functions are another widely-used class of orthonormal basis functions for $L_2([0,1])$. 
Let $(\psi_{jk})_{j\in\mathbb{N},k\in I_j}$ be an orthonormal basis of compactly supported wavelets for $L_2([0,1])$, 
with $j$ referring to the so-called ``resolution level'', and $k$ to the ``dilation'' (see, for example, Section E.3 in \cite{ghosal2017fundamentals}). We adopt the convention that the index set $I_j$ for the $j$th resolution level runs through $\{0,1,\cdots,2^j-1\}$. 
The exact definition and specific formulas for the wavelet basis are not of great interest to us, and for a complete and thorough review of wavelets from a statistical perspective, we refer to \cite{gine2015mathematical}. 
We shall assume that the wavelet basis $\psi_{jk}$'s are appropriately selected such that for any $f(x) = \sum_{j=0}^\infty\sum_{k\in I_j}\beta_{jk}\psi_{jk}(x)$, the following inequalities hold \citep{cohen2003numerical,cohen1993wavelets,hoffmann2015adaptive}:
\begin{align}
\|f\|_2
\leq 
\sum_{j=0}^\infty\left(\sum_{k\in I_j}\beta_{jk}^2\right)^{1/2}
\quad\text{and}\quad
\|f\|_\infty
\leq 
\sum_{j=0}^\infty2^{j/2}\max_{k\in I_j}|\beta_{jk}|\nonumber.
\end{align}

Write $f$ in terms of the wavelet series expansion $f(x) = \sum_{j=0}^\infty\sum_{k\in I_j}\beta_{jk}\psi_{jk}(x)$. The block prior for the wavelet series is then introduced through the wavelet coefficients $\beta_{jk}$'s as follows:
\begin{align}
\bbeta_j\mid A_j&\sim \mathrm{N}(\zero, A_k\eye_{n_k}),\quad A_j\sim g_j,\quad\text{independently for each }j,\nonumber
\end{align}
where $\bbeta_j=(\beta_{jk}:k\in I_j)$, $n_k = |I_k|=2^j$, and $g_j$ is given by
\begin{align}
g_j(t)&=\left\{
\begin{array}{ll}
\mathrm{e}^{{j^2}\log2}(\mathrm{e}^{-2^j\log 2}-T_j)t+T_j,&\quad 0\leq t\leq \mathrm{e}^{{-j^2}\log2},\\
\mathrm{e}^{-2^j\log2},&\quad \mathrm{e}^{-j^2\log2}<t\leq \mathrm{e}^{-j\log2},\\
0,&\quad t>\mathrm{e}^{-j\log2},
\end{array}
\right.\nonumber\\
T_j&= \exp\left[(1+j^2)\log 2\right]-\exp\left[{(-2^j+j^2-j)\log2}\right]+\mathrm{e}^{-2^j\log 2}.\nonumber
\end{align}
We further assume that $f_0$ is an $\alpha$-Sobolev function. 
For the block prior regression via wavelet series, the rate-exact Bayesian adaptation also holds.
\begin{theorem}\label{thm:block_prior_contraction_wavelet}
Suppose the true regression function $f_0\in\calH_\alpha(Q)$ for some $\alpha>1/2$ and $Q>0$, and $f(x)=\sum_{j=0}^\infty\sum_{k\in I_j}\beta_{jk}\psi_{jk}(x)$ is imposed the block prior for wavelet series $\Pi$ as described above.
Then there exists some sufficiently large constant $M>0$ such that 
\[
\mathbb{E}_0\left[\Pi\left(\|f-f_0\|_2>Mn^{-\alpha/(2\alpha+1)}\mid\calD_n\right)\right]\to 0.
\] 
\end{theorem}


\subsection{Beyond Fourier series: Gaussian spline regression} 
\label{sub:conditionally_gaussian_spline_regression}
The previous two examples show that the proposed framework can be applied to   priors through Fourier series expansions. The framework also works for prior distributions that are beyond Fourier basis, and we provide an example in this subsection. 

  Spline functions or splines are defined piecewise by polynomials
 on subintervals $[t_0, t_1), [t_1, t_2),\ldots,[t_{K - 1}, t_K]$ that are also globally smooth on the entire domain $[t_0, t_K]$, where $K$ is the number of subintervals. Without loss of generality, we may further require that $t_0 = 0$ and $t_K = 1$. A spline function is said to be of order $q$ for some positive integer $q$, if the involved polynomials are of degrees at most $q - 1$. Given the order $q$ and the number of subintervals $K$, the space of spline functions forms a linear space with dimension $m = q + K - 1$, and a basis for this linear space is also a set of spline functions, referred to as \emph{B-splines} and  denoted as $(B_k)_{k = 1}^{m}$. We refer the readers to \cite{de1978practical} and \cite{schumaker2007spline} for a systematic introduction of the spline functions. We present the following facts regarding the approximation property of splines, and the norm equivalence between spline functions and the coefficients of the corresponding B-splines. These results can be found in \cite{de1978practical}.

%

\begin{lemma}
\label{lemma:splines_lemma}
Let $f_0$ be an $\alpha$-H\"older function with $\alpha > 0$, and $q\geq \alpha$. Then there exists a constant $C$ depending on $f_0$, $q$, and $\alpha$, and $(\beta_{0k})_{k = 1}^m\subset\mathbb{R}$ with $m = q + K - 1$, such that
\[
\left\|\sum_{k = 1}^m\theta_kB_k(x) - f_0(x)\right\|_\infty\leq Cm^{-\alpha}.
\]
Furthermore, for any $\beta_1,\ldots,\beta_m\in\mathbb{R}$, 
\[
\max_{1\leq k\leq m}|\beta_k|
\asymp
\left\|\sum_{k = 1}^m\beta_kB_k(x)\right\|_\infty,
\quad
\left(\sum_{k = 1}^m\beta_k^2\right)^{1/2}\asymp
\sqrt{m}\left\|\sum_{k = 1}^m\beta_kB_k(x)\right\|_2.
\]
\end{lemma}
Assume that the true regression function $f_0$ is an $\alpha$-H\"older function with $\alpha > 1/2$. We now present the Gaussian spline prior, which simplifies the version presented in \cite{de2012adaptive}.  Assume that $f$ is a spline function of order $q \geq \alpha$ on $[0, 1]$, and the number of subintervals is $K$. Then we write 
$f$ in terms of a linear combination of the B-splines $f(x) = \sum_{k = 1}^m\beta_kB_k(x)$. The Gaussian spline prior assigns a prior distribution on $f$ by letting the coefficients $\beta_1,\ldots,\beta_m$ independently follow the standard normal distribution $\mathrm{N}(0, 1)$. Allowing $m$ grows moderately with the sample size $n$, we show in the following theorem that the posterior contraction rate with respect to $\|\cdot\|_2$ is minimax-optimal up to a logarithmic factor. 
\begin{theorem}\label{thm:Gaussian_spline_contraction}
Suppose the true regression function
$f_0\in\mathfrak{C}^\alpha(Q)$ for some $\alpha>1/2$ and $Q>0$, and $f(x)$ is assigned the  Gaussian spline prior $\Pi$ as described above with order $q\geq \alpha$ and the number of subintervals $K$.
If $m = q + K - 1= n^{1/(2\alpha + 1)}$, then there exists some sufficiently large constant $M>0$ such that 
\[
\mathbb{E}_0\left[\Pi\left(\|f-f_0\|_2>Mn^{-\alpha/(2\alpha+1)}(\log n)^{1/2}\mid\calD_n\right)\right]\to 0.
\] 
\end{theorem}
We briefly compare the result of Theorem \ref{thm:Gaussian_spline_contraction} with that in \cite{de2012adaptive}, which considered the  empirical $L_2$-distance and obtained the minimax-optimal rate up to a logarithmic factor in the fixed-design regression problem.  In contrast, we put the prior models in the context of the random-design regression and use the integrated $L_2$-distance, which can be viewed as a complement of the contraction result  presented in \cite{de2012adaptive}.

%


\section{Extensions} 
\label{sec:extensions}

     
\subsection{Extension to the fixed design regression} 
\label{sub:gaussian_random_series_regression_with_exact_rate} 
So far, the design points $(\bx_i)_{i=1}^n$ in this paper  are assumed to be randomly sampled from $[0,1]^p$. This is referred to as the random-design regression problem. There are, however, many cases where the design points $(\bx_i)_{i=1}^n$ are fixed and can be controlled. One of the examples is the design and analysis of computer experiments \cite{currin1991bayesian,sacks1989design}. To emulate a computer model, the design points are typically manipulated so that they are reasonably spread. In some physical experiments the design points can also be required to be fixed \cite{tuo2016theoretical}. 
In this subsection we show that by
slightly extending the framework in Section \ref{sec:the_framework_and_main_results},
the integrated $L_2$-distance contraction is also obtainable under similar conditions when the design points are reasonably selected.

Suppose that the design points $(x_i)_{i=1}^n\subset[0,1]$ are one-dimensional and fixed. 
Intuitively, the design points need to be relatively ``spread'' so that the global behavior of the true signal $f_0$ can be recovered as much as possible. Formally, we require that the design points satisfy 
\begin{align}\label{eqn:fixed_design_condition}
\sup_{x\in[0,1]}\left|\frac{1}{n}\sum_{i=1}^n\mathbbm{1}(x_i\leq x)-x\right|=O\left(\frac{1}{n}\right).
\end{align}
A simple example of such design  is the univariate equidistance design, \emph{i.e.}, $x_i = (i-1/2)/n$ (see, for example, \cite{yoo2016supremum,yoo2017adaptive}). 

Now we extend the framework in Section \ref{sec:the_framework_and_main_results} to the (one-dimensional) fixed-design regression problem. Specifically, this amounts to modifying the requirement for the sieve in \eqref{eqn:sieve_property}:
\begin{align}\label{eqn:sieve_fixed_design}
f,f_1&\in\calF_{m}(\delta)\Rightarrow\left|\mathbb{P}_n(f - f_0)^2 - \|f - f_0\|_2^2\right|
\leq \eta\left(\frac{m}{n}\|f - f_0\|_2^2 + \frac{\delta}{\sqrt{n}}\|f - f_0\|_2\right)
,\nonumber\\
\text{and}&
\left|\mathbb{P}_n(f - f_1)^2 - \|f - f_1\|_2^2\right|
\leq \eta\left(\frac{m}{n}\|f - f_1\|_2^2 + \frac{\delta}{\sqrt{n}}\|f - f_1\|_2\right).
\end{align}

With the above ingredients, we   present the following modification of Theorem \ref{thm:generic_contraction} for the fixed-design regression, which might be of independent interest as well. 
\begin{theorem}[Generic Contraction, Fixed-design]\label{thm:generic_contraction_fixed_design}
Suppose the design points $(x_i)_{i=1}^n$ are fixed and satisfy \eqref{eqn:fixed_design_condition}. 
Let $(\eps_n)_{n=1}^\infty$ and $(\leps_n)_{n=1}^\infty$ be sequences such that $\min(n\eps_n^2,n\leps_n^2)\to\infty$ as $n\to\infty$ with $0\leq\leps_n\leq\eps_n\to 0$. Let the sieves $(\calF_{m_n}(\delta))_{n=1}^\infty$ satisfy \eqref{eqn:sieve_fixed_design} for some constant $\delta>0$, where $m_n\to\infty$ and $m_n/n\to 0$. Suppose the conditions \eqref{eqn:summability}, \eqref{eqn:prior_mass_on_sieve}, and $\Pi(\|f-f_0\|_\infty<\leps_n)\geq\exp(-Dn\leps_n^2)$ hold for some constant $D>0$ and sufficiently large $n$ and $M$. 
Then 
\[
\mathbb{E}_0\left[\Pi(\|f-f_0\|_2>M\eps_n\mid\calD_n)\right]\to 0.
\]
\end{theorem}

As a sample application of the above framework, we consider one of the most popular Gaussian processes $\mathrm{GP}(0,K)$ with the covariance function $K(x,x') = \exp\left[-(x-x')^2\right]$ of the squared-exponential form  \cite{rasmussen2006gaussian}. We show  that optimal rates of contraction with respect to the integrated $L_2$-distance is also attainable when the design points are reasonably selected, in contrast to most Bayesian literatures that obtain rates of contraction with respect to the empirical $L_2$-distance. 


Given constants $c,Q>0$, we assume that the underlying true regression function $f_0$ lies in the following function class
\begin{align}
\calA_c(Q)=\left\{f(x)=\sum_{k=1}^\infty \beta_k\psi_k(x):
\sum_{k=1}^\infty \beta_k^2\exp\left(\frac{k^2}{c}\right)\leq Q^2
\right\}\nonumber.
\end{align}
The function class $\calA_c(Q)$ is closely related to the reproducing kernel Hilbert space (RKHS) associated with $\mathrm{GP}(0,K)$. 
For a complete and thorough review of RKHS from a Bayesian perspective, we refer to \cite{van2008reproducing}. 
A key feature of the functions in $\calA_c(Q)$ is that they are ``supersmooth'', \emph{i.e.}, they are infinitely differentiable. 
For the squared-exponential Gaussian process regression, the following property regarding the corresponding RKHS is available by applying Theorem 4.1 in \cite{van2008reproducing}.
\begin{lemma}
Let $\mathbb{H}$ be the RKHS associated with the squared-exponential Gaussian process $\mathrm{GP}(0,K)$, where $K(x,x')=\exp[-(x-x')^2]$. Then $\calA_4(Q)\subset\mathbb{H}$ for any $Q>0$. 
\end{lemma}
Under the squared-exponential Gaussian process prior $\Pi$, the rate of contraction of a supersmooth $f_0\in\calA_4(Q)$ is $1/\sqrt{n}$ up to a logarithmic factor.
\begin{theorem}
\label{thm:Gaussian_series_contraction}
Assume that the design points $(x_i)_{i = 1}^n$ are fixed and satisfy \eqref{eqn:fixed_design_condition}. 
Suppose the true regression function $f_0\in\calA_4(Q)$ for some $Q>0$, 
and $f$ follows the {squared-exponential Gaussian process prior $\Pi$}.
Then there exists some sufficiently large constant $M>0$ such that 
\[
\mathbb{E}_0\left[\Pi\left(\|f-f_0\|_2>M{n^{-1/2}(\log n)}\mid\calD_n\right)\right]\to 0.
\]
\end{theorem} 
\begin{remark}
For the squared-exponential Gaussian process regression with random design, the rate of contraction with respect to the integrated $L_2$-distance for $f_0\in\calA_4(Q)$ has been studied in the literature. 
In contrast, we remark that for the fixed-design regression problem, Theorem \ref{thm:Gaussian_series_contraction} is new and original, and provides a stronger result compared to the existing literature (see, for example, Theorem 10 in \cite{vaart2011information}).
\end{remark}



\subsection{Extension to sparse additive models in high dimensions} 
\label{sec:extension_to_sparse_additive_models_in_high_dimensions}

We have so far considered that the design space is low dimensional with fixed $p$. 
Nonetheless, the rapid development of technology has been enabling scientists to collect data with high-dimensional covariates, where the number of covariates $p$ can be much larger than the sample size $n$, to explore the potentially nonlinear relationship between these covariates and certain outcome of interest. The emergence of high dimensional prediction problems naturally motivates the study of nonparametric regression in high dimensions \cite{yang2015}.
In this section, we focus on one class of high-dimensional nonparametric regression problem, known as \emph{sparse additive models}, and illustrate that with suitable prior specification, the framework for low-dimensional Bayesian nonparametric regression naturally extends to such a high-dimensional scenario. 

We first review some background regarding the sparse additive models. 
Consider the additive regression model $y_i = f(\bx_i) + e_i$, where the regression function $f(\bx_i)$ is of an additive structure of the covariates $f(\bx_i) = \mu + \sum_{j = 1}^p f_j(x_{ij})$. Without loss of generality, one can assume that each component $f_j(x_j)$ is centered: $\int_0^1f_j(x_j)\mathrm{d}x_j = 0$, $j = 1,\cdots,p$. 
For sparse additive models in high dimensions, the number of covariates $p$ is typically much larger than the sample size $n$, and the underlying true regression function $f_0$ depends only on a small number of covariates, say, $x_{j_1},\ldots, x_{j_q}$, \emph{i.e.}, $f_0$ is of a sparse additive structure $f_0(\bx_i) = \mu_0 + \sum_{r = 1}^q f_{0j_r}(x_{ij_r})$, where each $f_{0r}:[0, 1]\to\mathbb{R}$ is a univariate function, and $q$ is the number of active covariates that does not change with sample size. Furthermore, the indices of these active covariates $\{j_1,\ldots, j_q\}$ and $q$ are unknown. 
This is referred to as the sparse additive models in high dimensions in the literature \cite{hastie2017generalized,koltchinskii2010,meier2009,doi:10.1111/j.1467-9868.2009.00718.x,raskutti2012minimax}. There have also been several works regarding Bayesian modeling of sparse additive models in high dimensions, see, for example, \cite{linero2017bayesian,rockova2017posterior,yang2015}.

To model the sparsity occurring in the high-dimensional additive regression model,
 we consider the following parameterization of $f$ by introducing the binary covariate-selection variables $z_1,\ldots,z_p\in\{0,1\}$:
\begin{align}\label{eqn:sparse_additive_model}
f(\bx) = \mu + \sum_{j = 1}^p z_jf_j(x_j),\quad   z_j\in\{0, 1\},\quad  j = 1,\cdots,p,
\end{align}
where $ z_j = 1$ indicates that the $j$th covariate is active and $ z_j = 0$ otherwise.
Following the strategy in Section \ref{sec:the_framework_and_main_results}, 
each component function $f_j$ is assigned a prior distribution independently across $j = 1,\ldots,p$. 
We complete the prior distribution $\Pi$ by imposing the selection variables $ z_j$ with a Bernoulli distribution $z_j\sim\mathrm{Bernoulli}(1/p)$. The Bernoulli prior for sparsity has been widely adopted in other high-dimensional Bayesian models with variable selection structures (see, for example, \cite{castillo2012,ročková2018,doi:10.1080/01621459.2016.1260469}). 

We now extend Theorem \ref{thm:generic_contraction} to sparse additive models by modifying the sieve property \eqref{eqn:sieve_property}. 
Denote $\bz = [z_1,\ldots,z_p]\transpose\in\{0, 1\}^p$ and let $A$ be a positive integer, 
we consider the sieve  $\calG_m^A(\delta) =  
\bigcup_{\substack{\|\bz\|_1\leq Aq
}}\calG_m(\delta,\bz)$,
where $\calG_m(\delta, \bz)$ with $\|\bz\|_1\leq A$ satisfies the following condition: there exists some constant $\eta > 0$ such that
\begin{align}
\label{eqn:sieve_property_additive}
f\in \calG_{m}(\delta, \bz) &\Longrightarrow
\|f - f_0\|_\infty^2\leq\eta(A^2m\|f - f_0\|_2^2 + \delta^2).
\end{align}
  \begin{theorem}[Generic Contraction, Sparse Additive Models]\label{thm:generic_contraction_additive}
  Consider the aforementioned sparse additive model in high dimensions.
  Let $(\eps_n)_{n=1}^\infty$ and $(\leps_n)_{n=1}^\infty$ be sequences such that $\min(n\eps_n^2,n\leps_n^2)\to\infty$ as $n\to\infty$, with $0\leq \leps_n\leq\eps_n\to 0$. Assume that there exist sieves of the form $\calG_{m_n}^{A_n}(\delta) = \bigcup_{\|\bz\|_1\leq A_nq}\calG_{m_n}(\delta,\bz)$, where $\calG_{m_n}(\delta, \bz)$ satisfies \eqref{eqn:sieve_property_additive}, $(m_n)_{n = 1}^\infty$, $(A_n)_{n =1 }^\infty$ are sequences such that $A_nm_n\eps_n^2\to0$, and $\delta$ is some constant.
  Let $(k_n)_{n = 1}^\infty$ be another sequence such that $k_n\leps_n^2 = O(1)$.
  Suppose the following conditions hold for some constants $\omega, D>0$ and sufficiently large $n$ and $M$:
  \begin{align}
  \label{eqn:summability_additive}
  &\sum_{j= M}^\infty N_{nj}^{A_n}\exp\left(-{Dnj^2\eps_n^2}\right)\to 0,\\
  \label{eqn:prior_mass_on_sieve_additive}
  &\Pi\left(\calG_{m_n}^{A_n}(\delta)^c\right)\lesssim\exp\left[-{\left(2D+\frac{1}{\sigma^2}\right)n\leps_n^2}\right],\\
  \label{eqn:prior_concentration_additive}
  &{\Pi\left(
  \widetilde B_n(k_n, \leps_n, \omega)
  \right)\geq\exp\left(-{Dn\leps_n^2}\right)},
  \end{align}
  where for any $m\in\mathbb{N}_+$ and $\eps, \omega > 0$,
  \begin{align*}
  &\widetilde B(m, \eps, \omega)
   = \left\{\|f - f_0\|_2\leq \eps, \|f - f_0\|_\infty^2\leq \eta'(m\|f - f_0\|_2^2 + \delta ^ 2)
  \right\}
  \end{align*} 
  for some constant $\eta' > 0$, and $N_{nj}^{A_n}=\calN(\xi j\eps_n,\calS_{nj}^{A_n}(\eps_n),\|\cdot\|_2)$ is the covering number of 
  \[\calS_{nj}^{A_n}=\{f\in\calG_{m_n}^{A_n}(\delta):j\eps_n<\|f-f_0\|_2\leq(j+1)\eps_n\}.\]
   Then $\mathbb{E}_0\left[\Pi(\|f-f_0\|_2>M\eps_n\mid\calD_n)\right]\to 0.$
  \end{theorem}
  The above framework is quite flexible and can be applied to a variety of prior models. For example, let us extend the finite random series prior discussed in Section \ref{sub:sieve_prior_regression_with_adaptive_rate} to the sparse additive models as follows: We model each component $f_j(x_j)$ via a Fourier series $f_j(x_j) = \sum_{k = 1}^\infty\beta_{jk}\psi_k(x_j)$, where $(\psi_k)_{k = 1}^\infty$ are the Fourier basis introduced in Section \ref{sub:sieve_prior_regression_with_adaptive_rate}. Then the coefficients $(\beta_{jk}:j = 1,\cdots,p,k = 2,3,\cdots)$ are assigned the following prior distributions:
  First sample an integer-valued random variable $N$ with density $\pi_N$ satisfying \eqref{eqn:sieve_prior_number_of_terms}, and then given $N = m$, sample the coefficients $(\beta_{jk}:j = 1,\cdots,p,k = 2, 3, \cdots)$ with
  \[
  \Pi(\mathrm{d}\beta_{jk}\mid N = m) = 
  \left\{\begin{array}{ll}
   g( \beta_{jk})\mathrm{d}\beta_{jk},&\quad\text{if }2\leq k\leq m,\\
  \delta_0(\mathrm{d}\beta_{jk}),&\quad\text{if }k > m,
  \end{array}\right.
  \]
  independently for all $j = 1, \cdots, p$ and $k = 2, 3, \cdots$, where $g(x) \propto \exp(-\tau_0|x|^\tau)$ for some $\tau, \tau_0 > 0$. Finally let $\mu$ follow a prior distribution with density $\pi(h)$ supported on $\mathbb{R}$, and set $\beta_{j1} = - \sum_{k = 2}^m\beta_{jk}\int_0^1\psi_k(x_j)\mathrm{d}x_j$, $j = 1,\cdots,p$ to ensure each component $f_j$ integrates to $0$. The prior specification is completed by combining the aforementioned Bernoulli prior on $\bz$. 

\begin{theorem}\label{thm:sieve_prior_contraction_additive}
  Consider the sparse additive model \eqref{eqn:sparse_additive_model} with the above prior distribution, and the dimension of the design space $p\geq 2$ possibly growing with the sample size $n$. Suppose the true regression function yields an additive structure: $f_0(\bx) = \mu + \sum_{r = 1}^qf_{0j_r}(x_{j_r})$, where each $f_{0j}\in\mathfrak{C}_\alpha(Q)$ for some $\alpha>1/2$ and $Q>0$ for all $j\in\{j_1,\cdots,j_q\}$, and $q$ does not change with $n$. Assume the dimension $p$ satisfies $\log p\lesssim \log n$.
  Let $f$ be imposed the prior $\Pi$ given above.
  Then there exists some sufficiently large constant $M>0$ such that 
  \[
  \mathbb{E}_0\left\{\Pi\left[\|f-f_0\|_2>Mn^{-\alpha/(2\alpha+1)}(\log n)^{t}\mathrel{\Big|}\calD_n\right]\right\}\to 0
  \]
  for any $t>\alpha/(2\alpha+1)$. 
\end{theorem}
\begin{remark}
The minimax rate of convergence with respect to $\|\cdot\|_2$ for sparse additive models is $n^{-\alpha/(2\alpha + 1)} + (\log p)/n$, provided that $\log p\lesssim n^c$ for some $c < 1$ \cite{yang2015}. The first term $n^{-\alpha/(2\alpha + 1)}$ is the usual rate for estimating a one-dimensional $\alpha$-smooth function, and the second term $(\log p)/n$ comes from the complexity of selecting the $q$ active covariates $x_{j_1},\cdots, x_{j_q}$ among $x_1,\cdots,x_p$. Under the assumption that $\log p\lesssim \log n$, the minimax rate of convergence is dominated by the first term $n^{-\alpha/(2\alpha + 1)}$. Thus Theorem \ref{thm:sieve_prior_contraction_additive} states that with the aforementioned finite random series prior for sparse additive model in high dimensions, the rate of contraction is also adaptive and minimax-optimal modulus an logarithmic factor, generalizing the result in Section \ref{sub:sieve_prior_regression_with_adaptive_rate}. 
\end{remark}



\section{Proofs of the main results} 
\label{sec:proofs}

  \begin{proof}[Proof of Lemma \ref{lemma:local_testability}]
  Recall the assumption  
  \begin{align}\label{eqn:sieve_norms_two_and_infinity}
  \|f-f_0\|_\infty^2\lesssim m\|f-f_0\|_2^2+\delta^2.
  \end{align} 
  Let us take $\xi = 1/(4\sqrt{2})$. 
  Define the test function to be 
  $\phi_n=\mathbbm{1}\left\{T_n > 0\right\}$, where
  \begin{align}
  T_n&=\sum_{i=1}^ny_i(f_1(\bx_i)-f_0(\bx_i))-\frac{1}{2}n\mathbb{P}_n(f_1^2-f_0^2)
  -\frac{\sqrt{n}}{8\sqrt{2}}\|f_1-f_0\|_2\sqrt{n\mathbb{P}_n(f_1-f_0)^2}\nonumber.
  \end{align}
  Before proceeding to the analysis of the type I and type II error probabilities, we introduce the motivation of the choice of $T_n$ as the test statistic. In fact, by the well-known Neyman-Pearson lemma, the most powerful test is the likelihood ratio test. In the context of random-design normal regression, the likelihood ratio test for testing $H_0:f = f_0$ against $H_A:f = f_1$ is equivalent to rejecting $f_0$ for large value of
  \begin{align*}
  &\exp\left\{-\frac{1}{2\sigma^2}\sum_{i = 1}^n[y_i - f_1(\bx_i)]^2 + \frac{1}{2\sigma^2}\sum_{i = 1}^n[y_i - f_0(\bx_i)]^2\right\}\\
  &\quad = \exp\left\{\frac{1}{\sigma^2}\sum_{i = 1}^ny_i(f_1(\bx_i) - f_0(\bx_i)) -  \frac{1}{2\sigma^2}n\mathbb{P}_n(f_1^2 - f_0^2)\right\},
  \end{align*}
  which is also equivalent to rejecting $f_0$ for large value of $T_n$. 

  We first consider the type I error probability. Under $\mathbb{P}_0$, we have $y_i=f_0(\bx_i)+e_i$, where $e_i$'s are i.i.d. $\mathrm{N}(0,\sigma^2)$ noises. Therefore, there exists a constant $C_1>0$ such that $\mathbb{P}_0(e_i>t)\leq\exp(-4C_1t^2)$ for all $t>0$. Then for a sequence $(a_i)_{i=1}^n\in\mathbb{R}^n$, the Chernoff bound yields 
  $\mathbb{P}_0\left(\sum_{i=1}^na_ie_i\geq t\right)\leq \exp\left(-{4C_1t^2}/{\sum_{i=1}^na_i^2}\right)$.
  Now we set $a_i=f_1(\bx_i)-f_0(\bx_i)$ and
  \[
  t=\frac{1}{2}n\mathbb{P}_n(f_1-f_0)^2+\frac{\sqrt{n}}{8\sqrt{2}}\|f_1-f_0\|_2\sqrt{n\mathbb{P}_n(f_1-f_0)^2}.
  \]
  Clearly, 
  \begin{align*}
  t^2&\geq n\mathbb{P}_n(f_1-f_0)^2\left[\frac{1}{4}n\mathbb{P}_n(f_1-f_0)^2+\frac{1}{128}n\|f_1-f_0\|_2^2\right]\\
  &\geq n\mathbb{P}_n(f_1-f_0)^2\left[\frac{1}{128}n\|f_1-f_0\|_2^2\right].
  \end{align*}
  Then under $\mathbb{P}_0(\cdot\mid\bx_1,\cdots,\bx_n)$, we have
  \begin{align}
  \mathbb{E}_0(\phi_n\mid\bx_1,\cdots,\bx_n)
  \leq \exp\left(-\frac{C_1}{32}n\|f_1-f_0\|_2^2\right)
  \nonumber.
  \end{align}
  It follows that the unconditioned error can be bounded:
  \begin{align}
  \mathbb{E}_0\phi_n\leq\exp\left(-\frac{C_1}{32}n\|f_1-f_0\|_2^2\right)\nonumber.
  \end{align}


  We next consider the type II error probability. Under $\mathbb{P}_f$, we have $y_i=f(\bx_i)+e_i$ with $e_i$'s being i.i.d. mean-zero Gaussian. Then for any $f$ with {$\|f-f_1\|_2\leq \|f_0-f_1\|_2/(4\sqrt{2})\leq\|f_0-f_1\|_2/4$}, we have
  \begin{align}
  \mathbb{E}_f(1-\phi_n)&\leq \mathbb{E}\left[\mathbbm{1}\left\{\mathbb{P}_n(f-f_1)^2\leq\frac{1}{16}\mathbb{P}_n(f_1-f_0)^2\right\}\mathbb{E}_f(1-\phi_n\mid\bx_1,\cdots,\bx_n)\right]
  \nonumber\\&\quad
  +\mathbb{P}\left(\mathbb{P}_n(f-f_1)^2>\frac{1}{16}\mathbb{P}_n(f_1-f_0)^2\right)\nonumber.
  \end{align}
  When $\mathbb{P}_n(f-f_1)^2\leq\mathbb{P}_n(f_1-f_0)^2/16$, we have
  \begin{align}
  &T_n+\frac{\sqrt{n}}{8\sqrt{2}}\|f_1-f_0\|_2\sqrt{n\mathbb{P}_n(f_1-f_0)^2}\nonumber\\
  &\quad= \sum_{i=1}^ne_i\left[f_1(\bx_i)-f_0(\bx_i)\right]+n\mathbb{P}_n(f-f_1)(f_1-f_0) + \frac{1}{2}n\mathbb{P}_n(f_1-f_0)^2\nonumber\\
  &\quad\geq \sum_{i=1}^ne_i\left[f_1(\bx_i)-f_0(\bx_i)\right]+\frac{1}{4}n\mathbb{P}_n(f_1-f_0)^2
  \nonumber.
  \end{align}
  Now set 
  \begin{align}
  R=R(\bx_1,\cdots,\bx_n) = \frac{1}{4}n\mathbb{P}_n(f_1-f_0)^2-\frac{\sqrt{n}}{8\sqrt{2}}\|f_1-f_0\|_2\sqrt{n\mathbb{P}_n(f_1-f_0)^2}\nonumber.
  \end{align}
  Then given $R\geq \sqrt{n}\|f_1-f_0\|_2\sqrt{n\mathbb{P}_n(f_1-f_0)^2}/(8\sqrt{2})$, we use the {Chernoff} bound to obtain
  \begin{align}
  \mathbb{P}_f\left(T_n<0\mid\bx_1,\cdots,\bx_n\right)&
  \leq\mathbb{P}\left(\sum_{i=1}^ne_i[f_1(\bx_i)-f_0(\bx_i)]\leq -R\mid\bx_1,\cdots,\bx_n\right)\nonumber\\
  &\leq\exp\left(-\frac{4C_1R^2}{n\mathbb{P}_n(f_1-f_0)^2}\right)
  \leq\exp\left(-\frac{C_1n\|f_1-f_0\|_2^2}{32}\right)\nonumber.
  \end{align}
  On the other hand,
  \begin{align}
  &\mathbb{P}\left(R<\frac{\sqrt{n}}{8\sqrt{2}}\|f_1-f_0\|_2\sqrt{n\mathbb{P}_n(f_1-f_0)^2}\right)
  =\mathbb{P}\left(\mathbb{G}_n(f_1-f_0)^2<-\frac{\sqrt{n}}{2}\|f_1-f_0\|_2^2\right)\nonumber.
  \end{align}
  It follows that
  \begin{align}
  &\mathbb{E}\left[\mathbbm{1}\left\{\mathbb{P}_n(f-f_1)^2\leq\frac{1}{16}\mathbb{P}_n(f_1-f_0)^2\right\}\mathbb{E}_f(1-\phi_n\mid\bx_1,\cdots,\bx_n)\right]\nonumber\\
  &\quad\leq\mathbb{E}\left[\mathbbm{1}\left\{R\geq\frac{\sqrt{n}}{8\sqrt{2}}\|f_1-f_0\|_2\sqrt{n\mathbb{P}_n(f_1-f_0)^2},
  \right.\right.\nonumber\\
  &\quad\qquad\qquad\left.\left.
  \mathbb{P}_n(f-f_1)^2\leq\frac{1}{16}\mathbb{P}_n(f_1-f_0)^2\right\}\mathbb{P}_f(T_n<0\mid\bx_1,\cdots,\bx_n)\right]
  \nonumber\\&\quad\quad
  +\mathbb{P}\left(R<\frac{\sqrt{n}}{8\sqrt{2}}\|f_1-f_0\|_2\sqrt{n\mathbb{P}_n(f_1-f_0)^2}\right)\nonumber\\
  &\quad\leq \exp\left(-\frac{C_1}{32}n\|f_1-f_0\|_2^2\right)+\mathbb{P}\left(\mathbb{G}_n(f_1-f_0)^2<-\frac{\sqrt{n}}{2}\|f_1-f_0\|_2^2\right)
  \nonumber.
  \end{align}
  Using Bernstein's inequality (Lemma 19.32 in \cite{van2000asymptotic}), we obtain the tail probability of the empirical process $\mathbb{G}_n(f_1-f_0)^2$
  \begin{align}
  &
  \mathbb{P}\left(\mathbb{G}_n(f_1-f_0)^2<-\frac{\sqrt{n}}{2}\|f_1-f_0\|_2^2\right)
  \nonumber\\&\quad
  \leq \exp\left(-\frac{1}{4}\frac{n\|f_1-f_0\|_2^4/4}{\mathbb{E}(f_1-f_0)^4+\|f_1-f_0\|_2^2\|f_1-f_0\|_\infty^2/2}\right)
  \nonumber\\
  &\quad
  {\leq\exp\left(-\frac{C'n\|f_1-f_0\|_2^2}{m\|f_1-f_0\|_2^2+\delta^2}\right),}
  \nonumber
  \end{align}
  {for some constant $C'>0$, 
  where we use the relation \eqref{eqn:sieve_norms_two_and_infinity}.}
  On the other hand, when {$\mathbb{P}_n(f-f_1)^2> \mathbb{P}_n(f_1-f_0)^2/16$}, we again use Bernstein's inequality and the fact that $f\in\{f\in\calF_m(\delta):\|f-f_1\|_2^2\leq2^{-5}\|f_0-f_1\|_2^2\}$ to compute
  \begin{align}
  &\mathbb{P}\left(\mathbb{P}_n(f-f_1)^2>\frac{1}{16}\mathbb{P}_n(f_1-f_0)^2\right)
  {\leq \exp\left(-\frac{1}{4}\frac{n\|f_1-f_0\|_2^4/1024}{\|g\|_2^2+\|f_1-f_0\|_2^2\|g\|_\infty/32}\right)}
  \nonumber,
  \end{align}
  {where $g=(f-f_1)^2-(f_1-f_0)^2/16$. We further compute
  \begin{align}
  \|g\|_2^2&\leq \left(\|(f-f_1)^2\|_2+\frac{1}{16}\|(f_1-f_0)^2\|_2\right)^2\nonumber\\
  &\leq\left(\|f-f_1\|_\infty\|f-f_1\|_2+\frac{1}{16}\|f_1-f_0\|_\infty\|f_1-f_0\|_2\right)^2\nonumber\\
  &\lesssim \|f-f_1\|_\infty^2\|f-f_1\|_2^2+\|f_1-f_0\|_\infty^2\|f_1-f_0\|_2^2\nonumber\\
  &\lesssim (m\|f_1-f_0\|_2^2+\delta^2)\|f_0-f_1\|_2^2,\nonumber
  \end{align}
  where we use \eqref{eqn:sieve_norms_two_and_infinity}, the fact that $\|f-f_1\|_2\lesssim \|f_0-f_1\|_2$, and that
  \begin{align}
  \|f-f_1\|_\infty^2&\leq 2\|f-f_0\|_\infty^2+2\|f_0-f_1\|_\infty^2
  \lesssim m\|f_1-f_0\|_2^2+\delta^2\nonumber.
  \end{align}
  Similarly, we obtain on the other hand, 
  \begin{align}
  \|g\|_\infty&
  =\|f-f_1\|_\infty^2+\frac{1}{16}\|f_1-f_0\|_\infty^2
  \lesssim m\|f_0-f_1\|_2^2+\delta^2.\nonumber
  \end{align}
  Therefore, we end up with}
  \[
  \mathbb{P}\left(\mathbb{P}_n(f-f_1)^2>\frac{1}{16}\mathbb{P}_n(f_1-f_0)^2\right)\leq\exp\left(-\frac{\tilde{C}_2n\|f_1-f_0\|_2^2}{m\|f_1-f_0\|_2^2+\delta^2}\right),
  \]
  where $\tilde{C}_2>0$ is some constant. 
  {Hence we obtain the following exponential bound for type I and type II error probabilities:
  \begin{align}
  \mathbb{E}_0\phi_n&\leq\exp(-Cn\|f_1-f_0\|_2^2)\nonumber,\\
  \mathbb{E}_{\fzcolor{}}(1-\phi_n)&\leq \exp(-Cn\|f_1-f_0\|_2^2)+2\exp\left(-\frac{Cn\|f_1-f_0\|_2^2}{m\|f_1-f_0\|_2^2+\delta^2}\right)\nonumber
  \end{align}
  for some constant $C>0$ whenever $\|f-f_1\|_2^2\leq \|f_1-f_0\|^2_2/32$. Taking the supremum of the type II error over $f\in\{f\in\calF_m(\delta):\|f-f_1\|_2^2\leq \|f_1-f_0\|^2_2/32\}$ completes the proof. 
  }
  \end{proof}

  \begin{proof}[Proof of Lemma \ref{lemma:global_testability}]
  We partition the alternative set into disjoint unions 
  \begin{align}
  &\left\{f\in\calF_m(\delta):\|f-f_0\|_2>M\eps_n\right\}\nonumber\\
  &\quad\subset\bigcup_{j= M}^\infty\left\{f\in\calF_m(\delta):j\eps_n<\|f-f_0\|_2\leq (j+1)\eps_n\right\}
  :=\bigcup_{j= M}^\infty\calS_{nj}(\eps_n).\nonumber
  \end{align}
  For each $S_{nj}(\eps_n)$, we can find $N_{nj}=\calN\left(\xi j\eps_n,\calS_{nj}(\eps_n),\|\cdot\|_2\right)$-many functions $f_{njl}\in\calS_{nj}(\eps_n)$ such that 
  \[
  \calS_{nj}(\eps_n)\subset\bigcup_{l=1}^{N_{nj}}\left\{f\in\calF_m(\delta):\|f-f_{njl}\|_2\leq\xi j\eps_n\right\}.
  \]
  Since for each $f_{njl}$, we have $f_{njl}\in\calS_{nj}(\eps_n)$, implying that $\|f_{njl}-f_0\|_2>j\eps_n$, we obtain the final decomposition of the alternative
  \[
  \calS_{nj}(\eps_n)\subset\bigcup_{l=1}^{N_{nj}}\left\{f\in\calF_m(\delta):\|f-f_{njl}\|_2\leq\xi \|f_0-f_{njl}\|_2\right\}.
  \]
  Now we apply Lemma \ref{lemma:local_testability} to construct individual test function $\phi_{njl}$ for each $f_{njl}$ satisfying the following property
  \begin{align}
  \mathbb{E}_0\phi_{njl}
  &\leq
  \exp(-Cnj^2\eps_n^2)\nonumber,\\
  \sup_{\{f\in\calF_m(\delta):\|f-f_{njl}\|_2^2\leq\xi^2\|f_0-f_{njl}\|^2_2\}}\mathbb{E}_f(1-\phi_{njl})&\leq\exp(-Cnj^2\eps_n^2)
  \nonumber\\&\quad
  +2\exp\left(-\frac{Cnj^2\eps_n^2}{mj^2\eps_n^2+\delta^2}\right),\nonumber
  \end{align}
  where we have used the fact that $\|f_{njl}-f_0\|_2> j\eps_n$.
  Now define the global test function to be $\phi_n=\sup_{j\geq M}\max_{1\leq l\leq N_{nj}}\phi_{njl}$. Then the type I error probability can be upper bounded using the union bound
  \begin{align}
  \mathbb{E}_0\phi_n&\leq \sum_{j=M}^\infty\sum_{l=1}^{N_{nj}}\mathbb{E}_0\phi_{njl}
  \leq\sum_{j=M}^\infty\sum_{l=1}^{N_{nj}}\exp(-Cnj^2\eps_n^2)
  =\sum_{j=M}^\infty N_{nj}\exp(-Cnj^2\eps_n^2)
  \nonumber.
  \end{align}
  The type II error probability can also be upper bounded:
  \begin{align}
  &\sup_{\{f\in\calF_m(\delta):\|f-f_0\|_2>M\eps_n\}}\mathbb{E}_f(1-\phi_n)
  \nonumber\\
  &\quad\leq \sup_{j\geq M}\sup_{l=1,\cdots,N_{nj}}\sup_{\{f\in\calF_m(\delta):\|f-f_{njl}\|_2\leq \xi\|f_0-f_{njl}\|_2\}}\mathbb{E}_f(1-\phi_{njl})\nonumber\\
  &\quad\leq\sup_{j\geq M}\sup_{l=1,\cdots,N_{nj}}\left[\exp(-Cj^2n\eps_n^2)+2\exp\left(-\frac{Cnj^2\eps_n^2}{mj^2\eps_n^2+\delta^2}\right)\right]\nonumber\\
  &\quad\leq\exp(-CM^2n\eps_n^2)+2\exp\left(-\frac{CnM^2\eps_n^2}{mM^2\eps_n^2+\delta^2}\right)\nonumber.
  \end{align}
  The proof is thus completed. 
  \end{proof}

  \begin{proof}[Proof of Lemma \ref{lemma:evidence_LB}]
  Denote the re-normalized restriction of $\Pi$ on\\ $B_n=B(k_n,\eps_n,\omega)$ to be $\Pi(\cdot\mid B_n)$, {and the random variables $(V_{ni})_{i=1}^n$, $(W_{ni})_{i=1}^n$ to be}
  \[
  V_{ni}=f_0(\bx_i)-\int f(\bx_i)\Pi(\mathrm{d}f\mid B_n),\quad
  W_{ni}=\frac{1}{2}\int (f(\bx_i)-f_0(\bx_i))^2\Pi(\mathrm{d}f\mid B_n).
  \]
  Let
  \[
  \calH_n:=\left\{\int\prod_{i =1 }^n\frac{p_f(\bx_i,y_i)}{p_0(\bx_i,y_i)}\Pi(\mathrm{d}f)>\Pi(B_n)\exp\left[-\left(C+\frac{1}{\sigma^2}\right){n\eps_n^2}\right]\right\}
  \]
  Then by Jensen's inequality
  \begin{align}
  \calH_n^c
  &\subset\left\{\int\prod_{i = 1}^n\frac{p_f(\bx_i,y_i)}{p_0(\bx_i,y_i)}\Pi(\mathrm{d}f\mid B_n)\leq\exp\left[-\left(C+\frac{1}{\sigma^2}\right){n\eps_n^2}\right]\right\}\nonumber\\
  &\subset\left\{\frac{1}{\sigma^2}\sum_{i=1}^n\left(e_iV_{ni}+W_{ni}\right)\geq\left(C+\frac{1}{\sigma^2}\right){n\eps_n^2}\right\}\nonumber\\
  &\subset\left\{\frac{1}{\sigma^2}\sum_{i=1}^ne_iV_{ni}\geq{Cn\eps_n^2}\right\}
  \cup
  \left\{\sum_{i=1}^nW_{ni}\geq n\eps_n^2\right\}
  \nonumber.
  \end{align}
  Now we use the Chernoff bound for Gaussian random variables to obtain the conditional probability bound for the first event given the design points $(\bx_i)_{i=1}^n$:
  \begin{align}
  \mathbb{P}_0\left(\sum_{i=1}^ne_iV_{ni}\geq C\sigma^2n\eps_n^2\mid\bx_1,\cdots,\bx_n\right)
  \leq\exp\left(-\frac{C^2\sigma^4n\eps_n^4}{\mathbb{P}_nV_{ni}^2}\right)\nonumber.
  \end{align}
  Since over the function class $B_n$, 
  we have $\|f-f_0\|_2\leq \eps_n$, $k_n\eps_n^2=O(1)$, and
  \begin{align}
  \|f-f_0\|_\infty^2&\lesssim k_n\|f - f_0\|_2^2 + \omega^2\leq k_n\eps_n^2 + \omega^2 = O(1)\nonumber,
  \end{align}
  it follows from Fubini's theorem that
  \begin{align}
  \mathbb{E}(V_{ni}^2)&\leq \int \|f_0-f\|_2^2\Pi(\mathrm{d}f\mid B_n)\leq \eps_n^2,
  \nonumber\\
  \mathbb{E}\left(V_{ni}^4\right)&\leq 
  \mathbb{E}\left[\int(f_0(\bx)-f(\bx))^4\Pi(\mathrm{d}f\mid B_n)\right]
  \nonumber\\&
  \leq\int\|f-f_0\|_\infty^2\|f-f_0\|_2^2\Pi(\mathrm{d}f\mid B_n)
  \lesssim \eps_n^2.\nonumber
  \end{align}
  Hence by the Chebyshev's inequality, 
  \begin{align}
  &\mathbb{P}\left(\left|\mathbb{P}_nV_{ni}^2-\mathbb{E}\left(V_{ni}^2\right)
  \right|>\eps_n^2\eps\right)
  \leq \frac{1}{n\eps^2\eps_n^4}\mathrm{var}(V_{ni}^2)\leq \frac{1}{n\eps_n^4\eps^2}\mathbb{E}(V_{ni}^4)
  \lesssim \frac{1}{n\eps_n^2}
  \to 0
  \nonumber
  \end{align}
  for any $\eps>0$, \emph{i.e.}, 
  $\mathbb{P}_nV_{ni}^2\leq\mathbb{E}V_{ni}^2+o_{P}(\eps_n^2)\leq \eps_n^2(1+o_P(1))$, and hence, 
  \[
  \exp\left(-\frac{C^2\sigma^4n\eps_n^4}{\mathbb{P}_nV_{ni}^2}\right)
  =\exp\left(-\frac{C^2\sigma^4n\eps_n^2}{1+o_P(1)}\right)
  \to 0
  \]
  in probability. Therefore by the dominated convergence theorem the unconditioned probability goes to $0$:
  \begin{align}
  \mathbb{P}_0\left(\sum_{i=1}^ne_iV_{ni}\geq C\sigma^2n\eps_n^2\right)
  &=\mathbb{E}\left[\mathbb{P}_0\left(\sum_{i=1}^ne_iV_{ni}\geq C\sigma^2n\eps_n^2\mid\bx_1,\cdots,\bx_n\right)\right]\nonumber\\
  &\leq 
  \mathbb{E}\left[\exp\left(-\frac{C^2\sigma^4n\eps_n^4}{\mathbb{P}_nV_{ni}^2}\right)\right]\to 0\nonumber.
  \end{align}
  For the second event we use the Bernstein's inequality. Since
  \begin{align}
  \mathbb{E}W_{ni}&
  =\frac{1}{2}\int\|f-f_0\|_2^2\Pi(\mathrm{d}f\mid B_n)\leq \frac{1}{2}\eps_n^2,\nonumber\\
  \mathbb{E}W_{ni}^2&
  \leq \frac{1}{4}\mathbb{E}\left[\int (f(\bx)-f_0(\bx))^4\Pi(\mathrm{d}f\mid B_n)\right]
  \leq\frac{1}{4}\int \|f-f_0\|_2^2\|f-f_0\|_\infty^2\Pi(\mathrm{d}f\mid B_n)\lesssim\eps_n^2\nonumber,
  \end{align}
  then
  \begin{align}
  \mathbb{P}\left(\sum_{i=1}^nW_{ni}>n\eps_n^2\right)&\leq
  \mathbb{P}\left(\sum_{i=1}^n(W_{ni}-\mathbb{E}W_{ni})>\frac{1}{2}n\eps_n^2\right)
  \nonumber\\&
  \leq \exp\left(-\frac{1}{4}\frac{n\eps_n^4/4}{\mathbb{E}W_{ni}^2+\eps_n^2\|W_{ni}\|_\infty/2}\right)\nonumber\\
  &\leq \exp\left(-\frac{\hat{C}_1n\eps_n^2}{1+\|W_{ni}\|_\infty}\right),\nonumber
  \end{align}
  where 
  $\|W_{ni}\|_\infty=\sup_{\bx\in[0,1]^p}({1}/{2})\int (f(\bx)-f_0(\bx))^2\Pi(\mathrm{d}f\mid B_n)$. 
  Since for any $f\in B_n$, $\|f-f_0\|_\infty=O(1)$, it follows that $\|W_{ni}\|_\infty=O(1)$, and hence, $\mathbb{P}(\sum_iW_{ni}>n\eps_n^2)\to 0$. To sum up, we conclude that
  \[
  \mathbb{P}(\calH_n^c)\leq \mathbb{P}_0\left(\sum_{i=1}^ne_iV_{ni}\geq C\sigma^2n\eps_n^2\right)+\mathbb{P}\left(\sum_{i=1}^nW_{ni}>n\eps_n^2\right)\to 0.
  \]
  \end{proof}

  \begin{proof}[Proof of lemma \ref{lemma:evidence_LB_infinity_norm}]
  Denote $\Pi(\cdot\mid B_n)=\Pi(\cdot\cap B_n)/\Pi(B_n)$ to be the re-normalized restriction of $\Pi$ on $B_n =\{\|f-f_0\|_\infty<\eps_n\}$, and
  \[
  V_{ni}=f_0(\bx_i)-\int f(\bx_i)\Pi(\mathrm{d}f\mid B_n),\quad
  W_{ni}=\frac{1}{2}\int (f(\bx_i)-f_0(\bx_i))^2\Pi(\mathrm{d}f\mid B_n).
  \]
  Similar to the proof of lemma \ref{lemma:evidence_LB}, we obtain
  \begin{align}
  \calH_n^c&=\left\{
  \int\prod_{i = 1}^n\frac{p_f(\bx_i,y_i)}{p_0(\bx_i,y_i)}\Pi(\mathrm{d}f)\leq\Pi(B_n)\exp\left[-\left(C+\frac{1}{\sigma^2}\right)n\eps_n^2\right]
  \right\}\nonumber\\
  &\subset\left\{\frac{1}{\sigma^2}\sum_{i=1}^n(e_iV_{ni}+W_{ni})\geq\left(C+\frac{1}{\sigma^2}\right)n\eps_n^2\right\}
  \nonumber\\&
  \subset\left\{\frac{1}{\sigma^2}\sum_{i=1}^ne_iV_{ni}\geq\left(C+\frac{1}{2\sigma^2}\right)n\eps_n^2\right\}\nonumber,
  \end{align}
  where we use the fact that $W_{ni}\leq(1/2)\int\|f-f_0\|_\infty^2\Pi(\mathrm{d}f\mid B_n)\leq \eps_n^2/2$ for all $f\in B_n$ in the last step. Conditioning on the design points $(\bx_i)_{i=1}^n$, we have
  \begin{align}
  \mathbb{P}_0(\calH_n^c\mid\bx_1,\cdots,\bx_n)
  &\leq\exp\left[-\left(C+\frac{1}{2\sigma^2}\right)^2\frac{\sigma^4n\eps_n^4}{\mathbb{P}_nV_{ni}^2}\right]\nonumber\\
  &\leq \exp\left\{-\left(C+\frac{1}{2\sigma^2}\right)^2\sigma^4n\eps_n^4\left[\int \|f-f_0\|_\infty^2\Pi(\mathrm{d}f\mid B_n)\right]^{-1}\right\}\nonumber\\
  &\leq\exp\left[-\left(C+\frac{1}{2\sigma^2}\right)^2\sigma^4n\eps_n^2\right]\to 0.\nonumber
  \end{align}
  The proof is completed by applying the dominated convergence theorem. 
  \end{proof}

  \begin{proof}[Proof of Theorem \ref{thm:generic_contraction}]
  For convenience denote the log-likelihood ratio function
  \[
  \Lambda_n(f) = \sum_{i = 1}^n[\log p_f(\bx_i,y_i) - \log p_0(\bx_i,y_i)].
  \]
  Let $\phi_n$ be the test function given by lemma \ref{lemma:global_testability}
   and 
  \[
  \calH_n=\left\{\int\exp(\Lambda_n(f\mid\calD_n))\Pi(\mathrm{d}f)\geq\exp\left[-\left(\frac{3D}{2}+\frac{1}{\sigma^2}\right)n\leps_n^2\right]\right\}.
  \]
  It follows from condition \eqref{eqn:prior_concentration} that
  \[
  \calH_n^c\subset\left\{
  \int\exp(\Lambda_n)\Pi(\mathrm{d}f)<\Pi(B_n(k_n,\leps_n,\omega))\exp\left[-\left(\frac{D}{2}+\frac{1}{\sigma^2}\right)n\leps_n^2\right]
  \right\},
  \]
  and hence, $\mathbb{P}_0(\calH_n^c)=o(1)$ by lemma \ref{lemma:evidence_LB}.
  Now we decompose the expected value of the posterior probability
  \begin{align}
  &\mathbb{E}_0\left[\Pi\left(\|f-f_0\|_2>M\eps_n\mid\calD_n\right)\right]\nonumber\\
  &\quad\leq\mathbb{E}_0\left[(1-\phi_n)\mathbbm{1}(\calH_n)\Pi(\|f-f_0\|_2>M\eps_n\mid\calD_n)\right]
  +\mathbb{E}_0\phi_n+\mathbb{E}_0[(1-\phi_n)\mathbbm{1}(\calH_n^c)]\nonumber\\
  &\quad\leq\mathbb{E}_0\left[(1-\phi_n)\mathbbm{1}(\calH_n)\frac{\int_{\{\|f-f_0\|_2>M\eps_n\}}\exp(\Lambda_n(f\mid\calD_n))\Pi(\mathrm{d}f)}{\int\exp(\Lambda_n(f\mid\calD_n))\Pi(\mathrm{d}f)}\right]
  \nonumber\\&\quad\quad
  +\mathbb{E}_0\phi_n+\mathbb{P}_0(\calH_n^c).\nonumber
  \end{align}
  By \eqref{eqn:summability} and lemma \ref{lemma:global_testability} the type I error probability $\mathbb{E}_0\phi_n\to 0$. 
  It suffices to bound the first term. Observe that on the event $\calH_n$, 
  the denominator in the square bracket can be lower bounded:
  \begin{align}
  &\mathbb{E}_0\left[(1-\phi_n)\mathbbm{1}(\calH_n)\frac{\int_{\{\|f-f_0\|_2>M\eps_n\}}\exp(\Lambda_n(f\mid\calD_n))\Pi(\mathrm{d}f)}{\int\exp(\Lambda_n(f\mid\calD_n))\Pi(\mathrm{d}f)}\right]\nonumber\\
  &\quad\leq {\exp\left[\left(\frac{3D}{2}+\frac{1}{\sigma^2}\right)n\leps_n^2\right]}
  \nonumber\\&\quad\quad\times
  \mathbb{E}_0\left[(1-\phi_n)\int_{\calF_{m_n}(\delta)\cap\{\|f-f_0\|_2>M\eps_n\}}\exp(\Lambda_n(f\mid\calD_n))\Pi(\mathrm{d}f)\right]\nonumber\\
  &\quad\quad+{\exp\left[\left(\frac{3D}{2}+\frac{1}{\sigma^2}\right)n\leps_n^2\right]}
  \mathbb{E}_0\left[\int_{\calF_{m_n}^c(\delta)}\exp(\Lambda_n(f\mid\calD_n))\Pi(\mathrm{d}f)\right]\nonumber.
  \end{align}
  By Fubini's theorem, lemma \ref{lemma:global_testability} we have
  \begin{align}
  &\mathbb{E}_0\left[(1-\phi_n)\int_{\calF_{m_n}(\delta)\cap\{\|f-f_0\|_2>M\eps_n\}}\exp(\Lambda_n(f\mid\calD_n))\Pi(\mathrm{d}f)\right]\nonumber\\
  &\quad\leq\int_{\calF_{m_n}(\delta)\cap\{\|f-f_0\|_2>M\eps_n\}}\mathbb{E}_0\left[(1-\phi_n)\exp(\Lambda_n\mid\calD_n)\right]\Pi(\mathrm{d}f)\nonumber\\
  &\quad\leq\sup_{f\in\calF_{m_n}(\delta)\cap\{\|f-f_0\|_2>M\eps_n\}}\mathbb{E}_f(1-\phi_n)\nonumber\\&\quad
  \leq
  \exp(-CM^2n\eps_n^2)+2\exp\left(-\frac{CM^2n\eps_n^2}{m_n\eps_n^2M^2+\delta}\right)\nonumber\\
  &\quad\leq \exp(-\tilde{C}M^2n\eps_n^2)\nonumber,
  \end{align}
  for some constatn $\tilde{C}>0$ for sufficiently large $n$, since $m_n\eps_n^2\to 0$ and $\delta=O(1)$ by assumption. 
  For the integral on $\calF_{m_n}^c(\delta)$, we apply Fubini's theorem to obtain
  \begin{align}
  \mathbb{E}_0\left[\int_{\calF_{m_n}^c(\delta)}\exp(\Lambda_n(f\mid\calD_n))\Pi(\mathrm{d}f)\right]
  &=\int_{\calF_{m_n}^c(\delta)}\mathbb{E}_0\left[\prod_{i=1}^n\frac{p_f(\bx_i,y_i)}{p_0(\bx_i,y_i)}\right]\Pi(\mathrm{d}f)
  \nonumber\\&
  \leq\Pi(\calF_{m_n}^c(\delta))\nonumber.
  \end{align}
  Hence we proceed to compute
  \begin{align}
  &\mathbb{E}_0\left[(1-\phi_n)\mathbbm{1}(\calH_n)\frac{\int_{\{\|f-f_0\|_2>M\eps_n\}}\exp(\Lambda_n(f\mid\calD_n))\Pi(\mathrm{d}f)}{\int\exp(\Lambda_n(f\mid\calD_n))\Pi(\mathrm{d}f)}\right]\nonumber\\
  &\quad\lesssim \exp\left[\left(\frac{3D}{2}+\frac{1}{\sigma^2}\right)n\leps_n^2-\tilde{C}M^2n\eps_n^2\right]
  \nonumber\\&\quad\quad
  +\exp\left[{\left(\frac{3D}{2}+\frac{1}{\sigma^2}\right)}n\leps_n^2-{\left(2D+\frac{1}{\sigma^2}\right)}n\eps_n^2\right]\to 0\nonumber
  \end{align}
  as long as $M$ is sufficiently large, where \eqref{eqn:prior_mass_on_sieve} is applied.
  \end{proof}










\begin{supplement}
\stitle{Supplement to ``A theoretical framework for Bayesian nonparametric regression''}
\sdescription{
The supplementary material contains the remaining proofs and additional technical results. 
}
\end{supplement}

\bibliographystyle{imsart-number}
\bibliography{reference}

\begin{thebibliography}{51}

\bibitem{arbel2013bayesian}
\begin{barticle}[author]
\bauthor{\bsnm{Arbel},~\bfnm{Julyan}\binits{J.}},
  \bauthor{\bsnm{Gayraud},~\bfnm{Ghislaine}\binits{G.}} \AND
  \bauthor{\bsnm{Rousseau},~\bfnm{Judith}\binits{J.}}
(\byear{2013}).
\btitle{Bayesian optimal adaptive estimation using a sieve prior}.
\bjournal{Scandinavian Journal of Statistics}
\bvolume{40}
\bpages{549--570}.
\end{barticle}
\endbibitem

\bibitem{belitser2003adaptive}
\begin{barticle}[author]
\bauthor{\bsnm{Belitser},~\bfnm{Eduard}\binits{E.}} \AND
  \bauthor{\bsnm{Ghosal},~\bfnm{Subhashis}\binits{S.}}
(\byear{2003}).
\btitle{Adaptive {B}ayesian inference on the mean of an infinite-dimensional
  normal distribution}.
\bjournal{The Annals of Statistics}
\bvolume{31}
\bpages{536--559}.
\end{barticle}
\endbibitem

\bibitem{canale2017posterior}
\begin{barticle}[author]
\bauthor{\bsnm{Canale},~\bfnm{Antonio}\binits{A.}} \AND
  \bauthor{\bsnm{De~Blasi},~\bfnm{Pierpaolo}\binits{P.}}
(\byear{2017}).
\btitle{Posterior asymptotics of nonparametric location-scale mixtures for
  multivariate density estimation}.
\bjournal{Bernoulli}
\bvolume{23}
\bpages{379--404}.
\end{barticle}
\endbibitem

\bibitem{castillo2012}
\begin{barticle}[author]
\bauthor{\bsnm{Castillo},~\bfnm{Ismaël}\binits{I.}} \AND
  \bauthor{\bparticle{van~der} \bsnm{Vaart},~\bfnm{Aad}\binits{A.}}
(\byear{2012}).
\btitle{Needles and Straw in a Haystack: Posterior concentration for possibly
  sparse sequences}.
\bjournal{Ann. Statist.}
\bvolume{40}
\bpages{2069--2101}.
\bdoi{10.1214/12-AOS1029}
\end{barticle}
\endbibitem

\bibitem{cohen2003numerical}
\begin{bbook}[author]
\bauthor{\bsnm{Cohen},~\bfnm{Albert}\binits{A.}}
(\byear{2003}).
\btitle{Numerical analysis of wavelet methods}
\bvolume{32}.
\bpublisher{Elsevier}.
\end{bbook}
\endbibitem

\bibitem{cohen1993wavelets}
\begin{barticle}[author]
\bauthor{\bsnm{Cohen},~\bfnm{Albert}\binits{A.}},
  \bauthor{\bsnm{Daubechies},~\bfnm{Ingrid}\binits{I.}} \AND
  \bauthor{\bsnm{Vial},~\bfnm{Pierre}\binits{P.}}
(\byear{1993}).
\btitle{Wavelets on the interval and fast wavelet transforms}.
\bjournal{Applied and computational harmonic analysis}
\bvolume{1}
\bpages{54--81}.
\end{barticle}
\endbibitem

\bibitem{currin1991bayesian}
\begin{barticle}[author]
\bauthor{\bsnm{Currin},~\bfnm{Carla}\binits{C.}},
  \bauthor{\bsnm{Mitchell},~\bfnm{Toby}\binits{T.}},
  \bauthor{\bsnm{Morris},~\bfnm{Max}\binits{M.}} \AND
  \bauthor{\bsnm{Ylvisaker},~\bfnm{Don}\binits{D.}}
(\byear{1991}).
\btitle{Bayesian prediction of deterministic functions, with applications to
  the design and analysis of computer experiments}.
\bjournal{Journal of the American Statistical Association}
\bvolume{86}
\bpages{953--963}.
\end{barticle}
\endbibitem

\bibitem{de1978practical}
\begin{barticle}[author]
\bauthor{\bparticle{de} \bsnm{Boor},~\bfnm{Carl}\binits{C.}}
(\byear{1978}).
\btitle{A practical guide to splines}.
\bjournal{Applied Mathematical Sciences, New York: Springer, 1978}.
\end{barticle}
\endbibitem

\bibitem{de2010adaptive}
\begin{barticle}[author]
\bauthor{\bsnm{De~Jonge},~\bfnm{R}\binits{R.}} \AND
  \bauthor{\bsnm{Van~Zanten},~\bfnm{JH}\binits{J.}}
(\byear{2010}).
\btitle{Adaptive nonparametric {B}ayesian inference using location-scale
  mixture priors}.
\bjournal{The Annals of Statistics}
\bvolume{38}
\bpages{3300--3320}.
\end{barticle}
\endbibitem

\bibitem{de2012adaptive}
\begin{barticle}[author]
\bauthor{\bsnm{De~Jonge},~\bfnm{R}\binits{R.}} \AND
  \bauthor{\bsnm{Van~Zanten},~\bfnm{JH}\binits{J.}}
(\byear{2012}).
\btitle{Adaptive estimation of multivariate functions using conditionally
  {G}aussian tensor-product spline priors}.
\bjournal{Electronic Journal of Statistics}
\bvolume{6}
\bpages{1984--2001}.
\end{barticle}
\endbibitem

\bibitem{gao2016rate}
\begin{barticle}[author]
\bauthor{\bsnm{Gao},~\bfnm{Chao}\binits{C.}} \AND
  \bauthor{\bsnm{Zhou},~\bfnm{Harrison~H}\binits{H.~H.}}
(\byear{2016}).
\btitle{Rate exact {B}ayesian adaptation with modified block priors}.
\bjournal{The Annals of Statistics}
\bvolume{44}
\bpages{318--345}.
\end{barticle}
\endbibitem

\bibitem{ghosal2000convergence}
\begin{barticle}[author]
\bauthor{\bsnm{Ghosal},~\bfnm{Subhashis}\binits{S.}},
  \bauthor{\bsnm{Ghosh},~\bfnm{Jayanta~K}\binits{J.~K.}} \AND
  \bauthor{\bsnm{Van Der~Vaart},~\bfnm{Aad~W}\binits{A.~W.}}
(\byear{2000}).
\btitle{Convergence rates of posterior distributions}.
\bjournal{The Annals of Statistics}
\bvolume{28}
\bpages{500--531}.
\end{barticle}
\endbibitem

\bibitem{ghosal2007posterior}
\begin{barticle}[author]
\bauthor{\bsnm{Ghosal},~\bfnm{Subhashis}\binits{S.}} \AND \bauthor{\bsnm{Van
  Der~Vaart},~\bfnm{Aad}\binits{A.}}
(\byear{2007}).
\btitle{Posterior convergence rates of Dirichlet mixtures at smooth densities}.
\bjournal{The Annals of Statistics}
\bvolume{35}
\bpages{697--723}.
\end{barticle}
\endbibitem

\bibitem{ghosal2007convergence}
\begin{barticle}[author]
\bauthor{\bsnm{Ghosal},~\bfnm{Subhashis}\binits{S.}} \AND \bauthor{\bsnm{Van
  Der~Vaart},~\bfnm{Aad}\binits{A.}}
(\byear{2007}).
\btitle{Convergence rates of posterior distributions for noniid observations}.
\bjournal{The Annals of Statistics}
\bvolume{35}
\bpages{192--223}.
\end{barticle}
\endbibitem

\bibitem{ghosal2017fundamentals}
\begin{bbook}[author]
\bauthor{\bsnm{Ghosal},~\bfnm{Subhashis}\binits{S.}} \AND
  \bauthor{\bparticle{van~der} \bsnm{Vaart},~\bfnm{Aad}\binits{A.}}
(\byear{2017}).
\btitle{Fundamentals of nonparametric {B}ayesian inference}
\bvolume{44}.
\bpublisher{Cambridge University Press}.
\end{bbook}
\endbibitem

\bibitem{ghosal2001entropies}
\begin{barticle}[author]
\bauthor{\bsnm{Ghosal},~\bfnm{Subhashis}\binits{S.}} \AND \bauthor{\bsnm{Van
  Der~Vaart},~\bfnm{Aad~W}\binits{A.~W.}}
(\byear{2001}).
\btitle{{Entropies and rates of convergence for maximum likelihood and Bayes
  estimation for mixtures of normal densities}}.
\bjournal{The Annals of Statistics}
\bvolume{29}
\bpages{1233--1263}.
\end{barticle}
\endbibitem

\bibitem{gine2015mathematical}
\begin{bbook}[author]
\bauthor{\bsnm{Gin{\'e}},~\bfnm{Evarist}\binits{E.}} \AND
  \bauthor{\bsnm{Nickl},~\bfnm{Richard}\binits{R.}}
(\byear{2015}).
\btitle{Mathematical foundations of infinite-dimensional statistical models}
\bvolume{40}.
\bpublisher{Cambridge University Press}.
\end{bbook}
\endbibitem

\bibitem{hastie2017generalized}
\begin{bincollection}[author]
\bauthor{\bsnm{Hastie},~\bfnm{Trevor~J}\binits{T.~J.}}
(\byear{2017}).
\btitle{Generalized additive models}.
In \bbooktitle{Statistical models in S}
\bpages{249--307}.
\bpublisher{Routledge}.
\end{bincollection}
\endbibitem

\bibitem{hoffmann2015adaptive}
\begin{barticle}[author]
\bauthor{\bsnm{Hoffmann},~\bfnm{Marc}\binits{M.}},
  \bauthor{\bsnm{Rousseau},~\bfnm{Judith}\binits{J.}} \AND
  \bauthor{\bsnm{Schmidt-Hieber},~\bfnm{Johannes}\binits{J.}}
(\byear{2015}).
\btitle{On adaptive posterior concentration rates}.
\bjournal{The Annals of Statistics}
\bvolume{43}
\bpages{2259--2295}.
\end{barticle}
\endbibitem

\bibitem{huang2004convergence}
\begin{barticle}[author]
\bauthor{\bsnm{Huang},~\bfnm{Tzee-Ming}\binits{T.-M.}}
(\byear{2004}).
\btitle{Convergence rates for posterior distributions and adaptive estimation}.
\bjournal{The Annals of Statistics}
\bvolume{32}
\bpages{1556--1593}.
\end{barticle}
\endbibitem

\bibitem{koltchinskii2010}
\begin{barticle}[author]
\bauthor{\bsnm{Koltchinskii},~\bfnm{Vladimir}\binits{V.}} \AND
  \bauthor{\bsnm{Yuan},~\bfnm{Ming}\binits{M.}}
(\byear{2010}).
\btitle{Sparsity in multiple kernel learning}.
\bjournal{Ann. Statist.}
\bvolume{38}
\bpages{3660--3695}.
\bdoi{10.1214/10-AOS825}
\end{barticle}
\endbibitem

\bibitem{kruijer2010adaptive}
\begin{barticle}[author]
\bauthor{\bsnm{Kruijer},~\bfnm{Willem}\binits{W.}},
  \bauthor{\bsnm{Rousseau},~\bfnm{Judith}\binits{J.}} \AND \bauthor{\bsnm{Van
  Der~Vaart},~\bfnm{Aad}\binits{A.}}
(\byear{2010}).
\btitle{{Adaptive Bayesian density estimation with location-scale mixtures}}.
\bjournal{Electronic Journal of Statistics}
\bvolume{4}
\bpages{1225--1257}.
\end{barticle}
\endbibitem

\bibitem{linero2017bayesian}
\begin{barticle}[author]
\bauthor{\bsnm{Linero},~\bfnm{Antonio~Ricardo}\binits{A.~R.}} \AND
  \bauthor{\bsnm{Yang},~\bfnm{Yun}\binits{Y.}}
(\byear{2017}).
\btitle{Bayesian regression tree ensembles that adapt to smoothness and
  sparsity}.
\bjournal{arXiv preprint arXiv:1707.09461}.
\end{barticle}
\endbibitem

\bibitem{meier2009}
\begin{barticle}[author]
\bauthor{\bsnm{Meier},~\bfnm{Lukas}\binits{L.}}, \bauthor{\bparticle{van~de}
  \bsnm{Geer},~\bfnm{Sara}\binits{S.}} \AND
  \bauthor{\bsnm{Bühlmann},~\bfnm{Peter}\binits{P.}}
(\byear{2009}).
\btitle{High-dimensional additive modeling}.
\bjournal{Ann. Statist.}
\bvolume{37}
\bpages{3779--3821}.
\bdoi{10.1214/09-AOS692}
\end{barticle}
\endbibitem

\bibitem{JMLR:v16:pati15a}
\begin{barticle}[author]
\bauthor{\bsnm{Pati},~\bfnm{Debdeep}\binits{D.}},
  \bauthor{\bsnm{Bhattacharya},~\bfnm{Anirban}\binits{A.}} \AND
  \bauthor{\bsnm{Cheng},~\bfnm{Guang}\binits{G.}}
(\byear{2015}).
\btitle{Optimal Bayesian Estimation in Random Covariate Design with a Rescaled
  Gaussian Process Prior}.
\bjournal{Journal of Machine Learning Research}
\bvolume{16}
\bpages{2837-2851}.
\end{barticle}
\endbibitem

\bibitem{doi:10.1111/j.1467-9868.2009.00718.x}
\begin{barticle}[author]
\bauthor{\bsnm{Pradeep},~\bfnm{Ravikumar}\binits{R.}},
  \bauthor{\bsnm{John},~\bfnm{Lafferty}\binits{L.}},
  \bauthor{\bsnm{Han},~\bfnm{Liu}\binits{L.}} \AND
  \bauthor{\bsnm{Larry},~\bfnm{Wasserman}\binits{W.}}
\btitle{Sparse additive models}.
\bjournal{Journal of the Royal Statistical Society: Series B (Statistical
  Methodology)}
\bvolume{71}
\bpages{1009-1030}.
\bdoi{10.1111/j.1467-9868.2009.00718.x}
\end{barticle}
\endbibitem

\bibitem{raskutti2012minimax}
\begin{barticle}[author]
\bauthor{\bsnm{Raskutti},~\bfnm{Garvesh}\binits{G.}},
  \bauthor{\bsnm{Wainwright},~\bfnm{Martin~J}\binits{M.~J.}} \AND
  \bauthor{\bsnm{Yu},~\bfnm{Bin}\binits{B.}}
(\byear{2012}).
\btitle{Minimax-optimal rates for sparse additive models over kernel classes
  via convex programming}.
\bjournal{Journal of Machine Learning Research}
\bvolume{13}
\bpages{389--427}.
\end{barticle}
\endbibitem

\bibitem{rasmussen2006gaussian}
\begin{bbook}[author]
\bauthor{\bsnm{Rasmussen},~\bfnm{Carl~Edward}\binits{C.~E.}} \AND
  \bauthor{\bsnm{Williams},~\bfnm{Christopher~KI}\binits{C.~K.}}
(\byear{2006}).
\btitle{Gaussian processes for machine learning}
\bvolume{1}.
\bpublisher{MIT press Cambridge}.
\end{bbook}
\endbibitem

\bibitem{rivoirard2012posterior}
\begin{barticle}[author]
\bauthor{\bsnm{Rivoirard},~\bfnm{Vincent}\binits{V.}} \AND
  \bauthor{\bsnm{Rousseau},~\bfnm{Judith}\binits{J.}}
(\byear{2012}).
\btitle{Posterior concentration rates for infinite dimensional exponential
  families}.
\bjournal{Bayesian Analysis}
\bvolume{7}
\bpages{311--334}.
\end{barticle}
\endbibitem

\bibitem{ročková2018}
\begin{barticle}[author]
\bauthor{\bsnm{Rockov\'a},~\bfnm{Veronika}\binits{V.}}
(\byear{2018}).
\btitle{Bayesian estimation of sparse signals with a continuous spike-and-slab
  prior}.
\bjournal{Ann. Statist.}
\bvolume{46}
\bpages{401--437}.
\bdoi{10.1214/17-AOS1554}
\end{barticle}
\endbibitem

\bibitem{doi:10.1080/01621459.2016.1260469}
\begin{barticle}[author]
\bauthor{\bsnm{Rockov\'a},~\bfnm{Veronika}\binits{V.}} \AND
  \bauthor{\bsnm{George},~\bfnm{Edward~I.}\binits{E.~I.}}
(\byear{2018}).
\btitle{The Spike-and-Slab LASSO}.
\bjournal{Journal of the American Statistical Association}
\bvolume{113}
\bpages{431-444}.
\bdoi{10.1080/01621459.2016.1260469}
\end{barticle}
\endbibitem

\bibitem{rockova2017posterior}
\begin{barticle}[author]
\bauthor{\bsnm{Rockov\'a},~\bfnm{Veronika}\binits{V.}} \AND \bauthor{\bsnm{{van
  der Pas}},~\bfnm{Stephanie}\binits{S.}}
(\byear{2017}).
\btitle{Posterior concentration for {B}ayesian regression trees and their
  ensembles}.
\bjournal{arXiv preprint arXiv:1708.08734}.
\end{barticle}
\endbibitem

\bibitem{sacks1989design}
\begin{barticle}[author]
\bauthor{\bsnm{Sacks},~\bfnm{Jerome}\binits{J.}},
  \bauthor{\bsnm{Welch},~\bfnm{William~J}\binits{W.~J.}},
  \bauthor{\bsnm{Mitchell},~\bfnm{Toby~J}\binits{T.~J.}} \AND
  \bauthor{\bsnm{Wynn},~\bfnm{Henry~P}\binits{H.~P.}}
(\byear{1989}).
\btitle{Design and analysis of computer experiments}.
\bjournal{Statistical science}
\bpages{409--423}.
\end{barticle}
\endbibitem

\bibitem{schumaker2007spline}
\begin{bbook}[author]
\bauthor{\bsnm{Schumaker},~\bfnm{Larry}\binits{L.}}
(\byear{2007}).
\btitle{Spline functions: {B}asic theory}.
\bpublisher{Cambridge University Press}.
\end{bbook}
\endbibitem

\bibitem{scricciolo2006convergence}
\begin{barticle}[author]
\bauthor{\bsnm{Scricciolo},~\bfnm{Catia}\binits{C.}}
(\byear{2006}).
\btitle{Convergence rates for {B}ayesian density estimation of
  infinite-dimensional exponential families}.
\bjournal{The Annals of Statistics}
\bvolume{34}
\bpages{2897--2920}.
\end{barticle}
\endbibitem

\bibitem{scricciolo2011posterior}
\begin{barticle}[author]
\bauthor{\bsnm{Scricciolo},~\bfnm{Catia}\binits{C.}}
(\byear{2011}).
\btitle{{Posterior rates of convergence for Dirichlet mixtures of exponential
  power densities}}.
\bjournal{Electronic Journal of Statistics}
\bvolume{5}
\bpages{270--308}.
\end{barticle}
\endbibitem

\bibitem{shen2013adaptive}
\begin{barticle}[author]
\bauthor{\bsnm{Shen},~\bfnm{Weining}\binits{W.}},
  \bauthor{\bsnm{Tokdar},~\bfnm{Surya~T}\binits{S.~T.}} \AND
  \bauthor{\bsnm{Ghosal},~\bfnm{Subhashis}\binits{S.}}
(\byear{2013}).
\btitle{{Adaptive Bayesian multivariate density estimation with Dirichlet
  mixtures}}.
\bjournal{Biometrika}
\bvolume{100}
\bpages{623--640}.
\end{barticle}
\endbibitem

\bibitem{stone1982optimal}
\begin{barticle}[author]
\bauthor{\bsnm{Stone},~\bfnm{Charles~J}\binits{C.~J.}}
(\byear{1982}).
\btitle{Optimal global rates of convergence for nonparametric regression}.
\bjournal{The Annals of Statistics}
\bpages{1040--1053}.
\end{barticle}
\endbibitem

\bibitem{tuo2016theoretical}
\begin{barticle}[author]
\bauthor{\bsnm{Tuo},~\bfnm{Rui}\binits{R.}} \AND
  \bauthor{\bsnm{Jeff~Wu},~\bfnm{CF}\binits{C.}}
(\byear{2016}).
\btitle{A theoretical framework for calibration in computer models:
  parametrization, estimation and convergence properties}.
\bjournal{SIAM/ASA Journal on Uncertainty Quantification}
\bvolume{4}
\bpages{767--795}.
\end{barticle}
\endbibitem

\bibitem{vaart2011information}
\begin{barticle}[author]
\bauthor{\bsnm{Vaart},~\bfnm{Aad van~der}\binits{A.~v.~d.}} \AND
  \bauthor{\bsnm{Zanten},~\bfnm{Harry~van}\binits{H.~v.}}
(\byear{2011}).
\btitle{Information rates of nonparametric Gaussian process methods}.
\bjournal{Journal of Machine Learning Research}
\bvolume{12}
\bpages{2095--2119}.
\end{barticle}
\endbibitem

\bibitem{van2000asymptotic}
\begin{bbook}[author]
\bauthor{\bparticle{Van~der} \bsnm{Vaart},~\bfnm{Aad~W}\binits{A.~W.}}
(\byear{2000}).
\btitle{Asymptotic statistics}
\bvolume{3}.
\bpublisher{Cambridge university press}.
\end{bbook}
\endbibitem

\bibitem{van2008rates}
\begin{barticle}[author]
\bauthor{\bparticle{van~der} \bsnm{Vaart},~\bfnm{Aad~W}\binits{A.~W.}} \AND
  \bauthor{\bparticle{van} \bsnm{Zanten},~\bfnm{J~Harry}\binits{J.~H.}}
(\byear{2008}).
\btitle{Rates of contraction of posterior distributions based on Gaussian
  process priors}.
\bjournal{The Annals of Statistics}
\bpages{1435--1463}.
\end{barticle}
\endbibitem

\bibitem{van2008reproducing}
\begin{bincollection}[author]
\bauthor{\bparticle{van~der} \bsnm{Vaart},~\bfnm{Aad~W}\binits{A.~W.}} \AND
  \bauthor{\bparticle{van} \bsnm{Zanten},~\bfnm{J~Harry}\binits{J.~H.}}
(\byear{2008}).
\btitle{Reproducing kernel {H}ilbert spaces of {G}aussian priors}.
In \bbooktitle{Pushing the limits of contemporary statistics: contributions in
  honor of Jayanta K. Ghosh}
\bpages{200--222}.
\bpublisher{Institute of Mathematical Statistics}.
\end{bincollection}
\endbibitem

\bibitem{van2009adaptive}
\begin{barticle}[author]
\bauthor{\bparticle{van~der} \bsnm{Vaart},~\bfnm{Aad~W}\binits{A.~W.}} \AND
  \bauthor{\bparticle{van} \bsnm{Zanten},~\bfnm{J~Harry}\binits{J.~H.}}
(\byear{2009}).
\btitle{Adaptive {B}ayesian estimation using a {G}aussian random field with
  inverse {G}amma bandwidth}.
\bjournal{The Annals of Statistics}
\bpages{2655--2675}.
\end{barticle}
\endbibitem

\bibitem{doi:10.1111/sjos.12159}
\begin{barticle}[author]
\bauthor{\bsnm{Weining},~\bfnm{Shen}\binits{S.}} \AND
  \bauthor{\bsnm{Subhashis},~\bfnm{Ghosal}\binits{G.}}
\btitle{Adaptive Bayesian Procedures Using Random Series Priors}.
\bjournal{Scandinavian Journal of Statistics}
\bvolume{42}
\bpages{1194-1213}.
\bdoi{10.1111/sjos.12159}
\end{barticle}
\endbibitem

\bibitem{xie2017bayesian}
\begin{barticle}[author]
\bauthor{\bsnm{Xie},~\bfnm{Fangzheng}\binits{F.}} \AND
  \bauthor{\bsnm{Xu},~\bfnm{Yanxun}\binits{Y.}}
(\byear{2017}).
\btitle{Bayesian Repulsive Gaussian Mixture Model}.
\bjournal{arXiv preprint arXiv:1703.09061}.
\end{barticle}
\endbibitem

\bibitem{Xie2017KEMPOL}
\begin{barticle}[author]
\bauthor{\bsnm{Xie},~\bfnm{Fangzheng}\binits{F.}} \AND
  \bauthor{\bsnm{Xu},~\bfnm{Yanxun}\binits{Y.}}
(\byear{2018}).
\btitle{Adaptive Bayesian Nonparametric Regression Using a Kernel Mixture of
  Polynomials with Application to Partial Linear Models}.
\bjournal{Bayesian Anal.}
\bnote{Advance publication}.
\bdoi{10.1214/19-BA1148}
\end{barticle}
\endbibitem

\bibitem{yang2017frequentist}
\begin{barticle}[author]
\bauthor{\bsnm{Yang},~\bfnm{Yun}\binits{Y.}},
  \bauthor{\bsnm{Bhattacharya},~\bfnm{Anirban}\binits{A.}} \AND
  \bauthor{\bsnm{Pati},~\bfnm{Debdeep}\binits{D.}}
(\byear{2017}).
\btitle{Frequentist coverage and sup-norm convergence rate in {G}aussian
  process regression}.
\bjournal{arXiv preprint arXiv:1708.04753}.
\end{barticle}
\endbibitem

\bibitem{yang2015}
\begin{barticle}[author]
\bauthor{\bsnm{Yang},~\bfnm{Yun}\binits{Y.}} \AND
  \bauthor{\bsnm{Tokdar},~\bfnm{Surya~T.}\binits{S.~T.}}
(\byear{2015}).
\btitle{Minimax-optimal nonparametric regression in high dimensions}.
\bjournal{Ann. Statist.}
\bvolume{43}
\bpages{652--674}.
\bdoi{10.1214/14-AOS1289}
\end{barticle}
\endbibitem

\bibitem{yoo2016supremum}
\begin{barticle}[author]
\bauthor{\bsnm{Yoo},~\bfnm{William~Weimin}\binits{W.~W.}} \AND
  \bauthor{\bsnm{Ghosal},~\bfnm{Subhashis}\binits{S.}}
(\byear{2016}).
\btitle{Supremum norm posterior contraction and credible sets for nonparametric
  multivariate regression}.
\bjournal{The Annals of Statistics}
\bvolume{44}
\bpages{1069--1102}.
\end{barticle}
\endbibitem

\bibitem{yoo2017adaptive}
\begin{barticle}[author]
\bauthor{\bsnm{Yoo},~\bfnm{William~Weimin}\binits{W.~W.}},
  \bauthor{\bsnm{Rousseau},~\bfnm{Judith}\binits{J.}} \AND
  \bauthor{\bsnm{Rivoirard},~\bfnm{Vincent}\binits{V.}}
(\byear{2017}).
\btitle{Adaptive Supremum Norm Posterior Contraction: Spike-and-Slab Priors and
  Anisotropic Besov Spaces}.
\bjournal{arXiv preprint arXiv:1708.01909}.
\end{barticle}
\endbibitem

\end{thebibliography}

\end{document}